\documentclass[]{amsart}
\pdfoutput=1
\usepackage[left=28truemm,right=28truemm,top=25truemm,bottom=25truemm]{geometry}
\usepackage[hidelinks,bookmarks=true,bookmarksnumbered=true,bookmarksopen=true,setpagesize=false]{hyperref}
\usepackage{caption}
\usepackage{amsmath,amssymb}
\usepackage{amsthm}
\usepackage{mathtools}
\mathtoolsset{showonlyrefs}
\theoremstyle{definition}
\usepackage{graphicx}
\usepackage{float}
\usepackage{bsvector}
\renewcommand{\forany}[1]{\ensuremath{{}^{\forall}\!#1}}
\newcommand{\sectionlab}[1]{\label{section:#1}}
\newcommand{\sectionref}[1]{\ref{section:#1}}
\newcommand{\subsectionlab}[1]{\label{subsection:#1}}
\newcommand{\subsectionref}[1]{Subsection~\ref{subsection:#1}}
\newcommand{\appendixlab}[1]{\label{appendix:#1}}
\newcommand{\appendixref}[1]{\ref{appendix:#1}}
\newtheorem{definition}{Definition}
\newcommand{\definitionlab}[1]{\label{definition:#1}}
\newcommand{\definitionref}[1]{Definition~\ref{definition:#1}}

\newtheorem{theorem}{Theorem}
\newcommand{\theoremlab}[1]{\label{theorem:#1}}
\newcommand{\theoremref}[1]{Theorem~\ref{theorem:#1}}
\newtheorem{corollary}{Corollary}
\newcommand{\corollarylab}[1]{\label{corollary:#1}}

\newtheorem{assumption}{Assumption}
\newcommand{\assumptionlab}[1]{\label{assumption:#1}}
\newcommand{\assumptionref}[1]{Assumption~\ref{assumption:#1}}

\newtheorem{remark}{Remark}
\newcommand{\remarklab}[1]{\label{remark:#1}}
\newcommand{\remarkref}[1]{Remark~\ref{remark:#1}}

\newcommand{\ff}[2]{\widetilde{f}(#1,#2)}
\newcommand{\hh}[2]{\widetilde{h}(#1,#2)}
\newcommand{\bregdiv}[3]{\mathcal{D}_{#1}({#2},{#3})}
\newcommand{\bigO}[1]{\mathcal{O}\left(#1\right)}
\renewcommand{\email}[1]{\texttt{(E-mail:~#1)}}

\title[Euler--Lagrange system for analyzing continuous-time accelerated gradient methods]%
{A unified Euler--Lagrange system for analyzing continuous-time accelerated gradient methods}
\author[M.~Toyoda]{Mitsuru Toyoda${}^1$}
\author[A.~Nishioka]{Akatsuki Nishioka${}^2$}
\author[M.~Tanaka]{Mirai Tanaka${}^3$}

\thanks{
	${}^1$Department of Mechanical Systems Engineering, 
	Tokyo Metropolitan University, 
	6-6 Asahigaoka, 
	Hino-shi, 
	191-0065 Tokyo, Japan \email{toyoda@tmu.ac.jp}}%
\thanks{
	${}^2$Department of Mathematical Informatics, 
	Graduate School of Information Science and Technology, 
	The University of Tokyo, 
	7-3-1 Hongo, 
	Bunkyo-ku, 
	113-8656 Tokyo, Japan \email{akatsuki\_nishioka@mist.i.u-tokyo.ac.jp}}%
\thanks{
	${}^3$Department of Fundamental Statistical Mathematics, 
	The Institute of Statistical Mathematics, 
	10-3 Midori-cho, 
	Tachikawa-shi, 
	190-8562 Tokyo, Japan, and 
	Continuous Optimization Team, 
	RIKEN Center for Advanced Intelligence Project, 
	Nihonbashi 1-chome Mitsui Building, 15th floor, 1-4-1 Nihonbashi, 
	Chuo-ku, 
	103-0027 Tokyo, Japan \email{mirai@ism.ac.jp}}
\begin{document}
\begin{abstract}
	This paper presents an Euler--Lagrange system
	for a continuous-time model of the accelerated gradient methods in smooth convex optimization
	and proposes an associated Lyapunov-function-based convergence analysis framework.
	Recently, ordinary differential equations (ODEs) with dumping terms have been developed to
	intuitively interpret the accelerated gradient methods, and
	the design of unified model describing the various individual ODE models have been examined.
	In existing reports, the Lagrangian, which results in the Euler-Lagrange equation, and the Lyapunov function for the convergence analysis
	have been separately proposed for each ODE.
	This paper proposes a unified Euler--Lagrange system and its Lyapunov function to cover the existing various models.
	In the convergence analysis using the Lyapunov function, a condition that parameters in the Lagrangian and Lyapunov function must satisfy is derived,
	and a parameter design for improving the convergence rate naturally results in the mysterious dumping coefficients.
	Especially, a symmetric Bregman divergence can lead to a relaxed condition of the parameters and
	a resulting improved convergence rate.
	As an application of this study, a slight modification in the Lyapunov function establishes the similar convergence proof
	for ODEs with smooth approximation in nondifferentiable objective function minimization.
\end{abstract}

\maketitle

\section{Introduction}
The accelerated gradient methods \cite{Ne18,Be17a} are variants of the gradient method
and performs theoretically and experimentally fast convergence by introducing
an additional state variable and associated technical coefficients in its updating equation.
Because of the improved convergence speed, it has been
widely applied in the various fields, including the image processing \cite{Be09b} and the model predictive control \cite{Pa14a,Ko15a}.
The aforementioned technical (and somewhat mysterious) updating equation of the original (i.e., discrete-time) accelerated gradient methods is derived by
a mathematical background in the convergence analysis and has an issue on the interpretability.
The intuitive explanation of the accelerated gradient methods has been tackled in existing reports \cite{Ah22},
and one of the simple representation is an ordinary differential equation (ODE) model
having a mysterious time-varying dumping term,
whose discretized model results in an accelerated gradient method
using a particular (and somewhat technical) discretization scheme \cite{Su16}.

The convergence properties, such as the convergence rate and the Lyapunov function,
in the continuous- and discrete-time domains need to be separately discussed.
More precisely, a discrete-time system (i.e., difference equation) obtained by the discretization
of a continuous-time system (i.e., ODE) depends on a discretization scheme;
therefore, a preferred convergence rate formula in the continuous-time domain does not
directly indicates that a resulting discrete-time system inherits similar convergence properties.
In particular, the scaling of the continuous-time axis results in
an arbitrary scaled expression of the convergence rate \cite{Us23}.
However, besides the ODE models can contribute to the better interpretability,
new ODE models and related discretization schemes have been expected to
develop new optimization methods and attracted much attention
in the fields of the mathematical optimization and numerical computation \cite{Ki23,Us23,To23a}.

The aforementioned system with the dumping term is interpreted as a mechanical system,
and its Lagrangian and resulting Euler--Lagrange equation have been recently explored \cite{Wi16,Wi21}.
On defining the Lagrangian,
the Euler--Lagrange equation is calculated using the partial derivatives of the Lagrangian \cite{Tr96,Ki23}
and has been utilized in the analysis of the mechanical systems (e.g., robotic systems \cite{Or89}) in the control theory.
The interpretable Lyapunov function, which describes a sort of the energy of a target system, is expected to be derived
by exploiting physical insights on the design of the Lagrangian, which is the difference of the kinetic and potential energies.
For existing Lyapunov functions individually proposed for each ODE model,
the design of unified Lagrangians and Lyapunov functions has been challenged in recent studies \cite{Wi16,Ki23,Us23}.
Motivated by the existing work summarized above,
this paper presents a unified Lagrangian and Lyapunov function generating the aforementioned various ODE models.
The main contributions of this paper are summarized as follows:
\begin{enumerate}
	\item By introducing a further generalized model of an existing unified Lagrangian and Lyapunov function than that proposed in \cite{Ki23},
	      the existing heuristic design of the Lyapunov function is simplified as a condition on the parameters in the proposed Lyapunov function.
	      On focusing on the improvement of the convergence rate, the existing technical (and mysterious) dumping coefficients are reasonably obtained.
	\item The convergence analysis with a symmetric Bregman divergence is performed.
	      On exploiting the symmetricity,
	      a relaxed condition and a wide class of Lyapunov functions, deriving preferred convergence properties, are available.
	      The analysis is a generalization of a technique improving the convergence rate by exploiting the Bregman divergence
	      with the $\ell_2$-norm \cite{At19}.
	\item The same convergence analysis follows from a simple modification in the Lagrangian and Lyapunov function for the smoothing technique in nonsmooth convex optimization problems.
	      This is a generalized result of a particular case of a non-strongly convex objective function and a dumping coefficient inversely proportional to the time \cite{Qu22}.
\end{enumerate}

The existing convergence rates are summarized in \tref{rates}.
The dumping terms that this paper addresses $D\dot{\bsx}(t)$,
$\left(\sqrt{\sigma}\left[3+\tanh^2\left([\sqrt{\sigma}/2]t\right)\right]\right)/
	\left(2\tanh\left([\sqrt{\sigma}/2]t\right)\right)\dot{\bsx}(t)$, and $(C/t)\dot{\bsx}(t)$
are discussed in \cite{Po87}, \cite{Ki23}, and \cite{Su16}, respectively.
The results in the second to forth rows of \tref{rates} is covered by that of the first row, given in this paper;
the corresponding sections are found in the first column of \tref{rates}.
\begin{table}[htbp]
	\centering
	\caption{Models and resulting convergence rates.}
	\tlab{rates}
	\footnotesize
	\begin{tabular}{llll}
		\hline
		Section                                                     & Model             & Damping                                & $f(\bsx(t)) - f(\bsx^*)$ \\\hline
		Sections \ref{section:Standard} and \ref{section:Symmetric} & Eq.~\eref{ODE}    & (see Eq.~\eref{ODE})                   &
		$\begin{array}{l} \mathcal{O}(\eexp{-\nu(t)})\end{array}$                                                                                    \\
		\subsectionref{TI}                                          & Eq.~\eref{StODE}  & $D\dot{\bsx}(t)$ (see \cite{Po87})     &
		$\begin{dcases}
				 \mathcal{O}\left(\exp\left(-[D/2]t\right)\right)                                            & (D \leq 2\sqrt{\sigma}), \\
				 \mathcal{O}\left(\exp\left(-\left[[D - \sqrt{D^2 - 4\sigma}]\middle/2\right]t\right)\right) & (D \geq 2\sqrt{\sigma}).
			 \end{dcases}$                            \\
		\subsectionref{hyp}                                         & Eq.~\eref{hypODE} &
		$\dfrac{\sqrt{\sigma}\left[3+\tanh^2\left([\sqrt{\sigma}/2]t\right)\right]}{2\tanh\left([\sqrt{\sigma}/2]t\right)}
		\dot{\bsx}(t)$ (see \cite{Ki23})                            &
		$\begin{dcases}
				 \bigO{1/\sinh^2\left([\sqrt{\sigma}/2]t\right)} & (\sigma > 0), \\
				 \bigO{t^{-2}}                                   & (\sigma = 0).
			 \end{dcases}$                                                             \\
		\subsectionref{C/t}                                         & Eq.~\eref{C/tODE} & $(C/t)\dot{\bsx}(t)$ (see \cite{Su16}) &
		$\begin{dcases}
				 \bigO{t^{-2C/3}} & (0 < C \leq 3),       \\
				 \bigO{t^{-2}}    & (3 \leq C < +\infty).
			 \end{dcases}$                                                                                    \\\hline
	\end{tabular}
\end{table}

\section{Preliminaries and Notations}
\begin{itemize}
	\item For notational simplicity, the substitution into a partial derivative (a column vector)
	      and its transpose (a row vector) are denoted by
	      \begin{equation}
		      \begin{dcases}
			      \nabla_\bsxi F(\bsxi_0,\bszeta_0) =  \left[\tpdif{F(\bsxi,\bszeta)}{\bsxi}\right]_{\bsxi = \bsxi_0, \bszeta = \bszeta_0}, \\
			      \nabla_\bsxi^\top F(\bsxi_0,\bszeta_0) = \left[\nabla_\bsxi F(\bsxi_0,\bszeta_0)\right]^\top  = \left[\tpdif{F(\bsxi,\bszeta)}{\bsxi^\top}\right]_{\bsxi = \bsxi_0, \bszeta = \bszeta_0},
		      \end{dcases}
	      \end{equation}
	      respectively. In the partial derivative of a function having just one argument, the subscript of the $\nabla$ operator is omitted.
	\item The Bregman divergence of two points $\bsx$ and $\bsy$ with respect to a convex function $h$ \cite{Wi16,Wi21} is defined by
	      \begin{equation}
		      \mathcal{D}_h(\bsy,\bsx) = h(\bsy) - h(\bsx) - \nabla^\top h(\bsx) (\bsy-\bsx) \geq 0.
	      \end{equation}
	      The Bregman three-point identity \cite{Wi16,Wi21}
	      \begin{equation}
		      [\nabla h(\bsx_2) - \nabla h(\bsx_3)]^\top[\bsx_1 - \bsx_2]
		      = - \mathcal{D}_h(\bsx_1,\bsx_2) + \mathcal{D}_h(\bsx_1,\bsx_3) - \mathcal{D}_h(\bsx_2,\bsx_3) \elab{threepoint}
	      \end{equation}
	      holds. The time derivative is calculated as follows:
	      \begin{equation}
		      \begin{split}
			       & \dif{\mathcal{D}_h(\bsy(t),\bsx(t))}{t}                                                                   \\
			       & = \nabla^\top h(\bsy(t))\dot{\bsy}(t) - \nabla^\top h(\bsx(t))\dot{\bsx}(t)
			      -\dif{\nabla^\top h(\bsx(t))}{t}(\bsy(t) - \bsx(t)) - \nabla^\top h(\bsx(t)) (\dot{\bsy}(t) - \dot{\bsx}(t)) \\
			       & = \dot{\bsy}^\top(t)[\nabla h(\bsy(t)) - \nabla h(\bsx(t))]
			      - \dif{\nabla^\top h(\bsx(t))}{t}(\bsy(t) - \bsx(t)).
		      \end{split}
		      \elab{bregdivdiff}
	      \end{equation}
	\item If an inequality $\bregdiv{f}{\bsy}{\bsx} \geq \sigma\bregdiv{h}{\bsy}{\bsx}\, (\sigma \geq 0)$ is satisfied,
	      $f$ is called $\sigma$-uniformly convex with respect to $h$ \cite{Wi21}.
	\item The hyperbolic functions are defined as follows: $\sinh(x) = (1/2)(\eexp{x}-\eexp{-x})$;
	      $\cosh(x) = (1/2)(\eexp{x}+\eexp{-x})$; and $\tanh(x) = \sinh(x)/\cosh(x) = (\eexp{x}-\eexp{-x})/(\eexp{x}+\eexp{-x})$.
	      An identity $\cosh^2(x) - \sinh^2(x) = 1$ holds.
	\item For a Lagrangian $\mathcal{L}(\bsx,\bsv,t)$,
	      the Euler--Lagrange equation (see, e.g., \cite{Or89}) is given by
	      \begin{equation}
		      \dif{}{t}\nabla_\bsv \mathcal{L}(\bsx(t),\dot{\bsx}(t),t) - \nabla_\bsx\mathcal{L}(\bsx(t),\dot{\bsx}(t),t)
		      = \bszero_{n_\bsx}.
		      \elab{EL}
	      \end{equation}
	\item For functions $f(t)$ and $g(t)$ on $[t_0,+\infty)$,
	      if there exist $C>0$ and $T\in[t_0,+\infty)$ such that $f(t) \leq Cg(t)$  for arbitrary $t \geq T$,
	      a notation $f(t) = \mathcal{O}(g(t))$ is used.
\end{itemize}
\section{Generalized Results using Bregman Divergence}\sectionlab{Standard}
This paper considers the continuous-time accelerated gradient methods for the following minimization problem:
\begin{equation}
	\minimize_{\bsx\in\R^{n_\bsx}}f(\bsx),
\end{equation}
where $f:\R^{n_\bsx}\to\R$ is assumed to be a differentiable convex function except for Section~\sectionref{smoothing}.
This section addresses a general case using the Bregman divergence as a metric of two points
and presents the Lagrangian, which specifies the resulting Euler--Lagrange equation,
and the Lyapunov function, playing an important role in the subsequent convergence analysis.
The following assumptions are placed throughout this paper.
\begin{assumption}\assumptionlab{0}
	\mbox{}
	\begin{itemize}
		\item $h$ is twice-differentiable and strongly convex.
		\item $f$ is differentiable (except for Section~\sectionref{smoothing}) $\sigma$-uniformly convex with respect to $h$ $(\sigma \geq 0)$.
		\item An ODE in each section has a solution on a considered time-interval.
	\end{itemize}
\end{assumption}
The existence of the solution has been reported for each model (e.g., that of the damping coefficient $C/t$ is in \cite{At19} and
the case of the damping using the hyperbolic functions is in \cite{Ki23}).
The subsequent analysis uses various time-varying parameters,
and their usage is summarized in \tref{params}.
\begin{table}[H]
	\centering
	\caption{Usage of the time-varying parameters.}
	\tlab{params}
	\footnotesize
	\begin{tabular}{lll}
		\hline
		Parameter   & Usage                                      & Examples                                                                                                                       \\\hline
		$\alpha(t)$ & Scaling in auxiliary variable              & $\bsz(t) = \bsx(t) + \eexp{-\alpha(t)}\dot{\bsx}(t)$                                                                           \\
		$\beta(t)$  & Convergence rate in existing studies       & $f(\bsx(t)) - f(\bsx^*) = \mathcal{O}(\eexp{-\beta(t)}), \, \dot{\beta}(t) \leq \eexp{\alpha(t)}$                              \\
		$\gamma(t)$ & Auxiliary parameter in existing studies    & $\dot{\gamma}(t) := \eexp{\alpha(t)}$ (Definition)                                                                             \\\hline
		$\delta(t)$ & Scaling in proposed Lagrangian             & $\mathcal{L}(\bsx,\bsv,t) = \eexp{\delta(t)}\left(\eexp{\eta(t)}\mathcal{D}_h(\bsx+\eexp{-\alpha(t)}\bsv,\bsx)-f(\bsx)\right)$ \\
		$\eta(t)$   & Scaling in proposed Lagrangian             & same as above                                                                                                                  \\
		$\nu(t)$    & Convergence rate in this study             & $f(\bsx(t)) - f(\bsx^*) = \mathcal{O}(\eexp{-\nu(t)}), \, \dot{\nu}(t) \leq \eexp{\alpha(t)}$                                  \\
		$\pi(t)$    & Parameter for symmetric Bregman divergence & $\eexp{\eta(t)+\pi(t)}\bregdiv{h}{\bsx(t)}{\bsx^*}$ in Lyapunov function
		\\		\hline
	\end{tabular}
\end{table}
Compared with existing reports \cite{Wi21,Ki23} using $\alpha(t), \beta(t)$, and $\gamma(t)$,
this paper uses $\alpha(t), \delta(t), \eta(t)$, and $\nu(t)$ for further generalized discussion.
The parameter $\pi(t)$ is only used in the case of the symmetric Bregman divergence.

For the aforementioned practical models (e.g., the damping coefficient $C/t$),
the Bregman divergence is introduced with the $\ell_2$-norm form $h(\cdot)=(1/2)\|\cdot\|_2^2$,
and the subsequent analysis of \subsectionref{ell2} sets the parameter $\eta(t)$ as $\eta(t):=2\alpha(t)$ (i.e., $\eta(t)$ depends on $\alpha(t)$.

\subsection{Lagrangian}
First, the Lagrangian deriving the continuous-time accelerated gradient methods is presented.
By using differentiable functions $\alpha(t), \delta(t)$, and $\eta(t)$ on $t\in[t_0,T]$,
a generalized form compared with the existing studies \cite{Wi21,Ki23} is given as follows:
\begin{equation}
	\mathcal{L}(\bsx,\bsv,t) = \eexp{\delta(t)}\bigl(\eexp{\eta(t)}\mathcal{D}_h(\bsx+\eexp{-\alpha(t)}\bsv,\bsx)-f(\bsx)\bigr).
	\elab{Lagrangian}
\end{equation}
\begin{theorem}\theoremlab{EL}
	Let an auxiliary state variable $\bsz(t) = \bsx(t) + \eexp{-\alpha(t)}\dot{\bsx}(t)$.
	Under \assumptionref{0}, the Euler--Lagrange equation associated with the Lagrangian \eref{Lagrangian}
	results in the following state equation of $[\bsx(t), \bsz(t)]^\top$:
	\begin{equation}
		\begin{dcases}
			\dot{\bsx}(t) = \eexp{\alpha(t)}(\bsz(t) - \bsx(t)), \\
			\nabla^2 h(\bsz(t))\dot{\bsz}(t) = -(\dot{\delta}(t) + \dot{\eta}(t) - \dot{\alpha}(t) - \eexp{\alpha(t)})[\nabla h(\bsz(t)) - \nabla h(\bsx(t))] - \eexp{-\eta(t)+\alpha(t)}\nabla f(\bsx(t)).
		\end{dcases}
		\elab{stateEq}
	\end{equation}
	The state equation above is expressed by the following ODE:
	\begin{equation}
		\begin{split}
			 & \ddot{\bsx}(t) + (\eexp{\alpha(t)} - \dot{\alpha}(t))\dot{\bsx}(t) \\
			 & +
			[\nabla^2 h(\bsz(t))]^{-1}
			\left(\eexp{\alpha(t)}(\dot{\delta}(t) +  \dot{\eta}(t) - \dot{\alpha}(t) - \eexp{\alpha(t)})[\nabla h(\bsz(t))- \nabla h(\bsx(t))]
			+\eexp{-\eta(t) + 2\alpha(t)}\nabla f(\bsx(t))\right) = \bszero_{n_\bsx}.
		\end{split}
		\elab{ODE}
	\end{equation}
\end{theorem}
\begin{proof}
	In principle, the substitution of the given Lagrangian \eref{Lagrangian} into the Euler--Lagrange equation \eref{EL} results in the claim of the theorem;
	however, for reader's convenience, the calculation is shown here. The Bregman divergence in the Lagrangian is rearranged as follows:
	\begin{equation}
		\begin{split}
			\mathcal{L}(\bsx,\bsv,t) & =
			\eexp{\delta(t)}\bigl(\eexp{\eta(t)}\mathcal{D}_h(\bsx+\eexp{-\alpha(t)}\bsv,\bsx)-f(\bsx)\bigr) \\
			                         & =\eexp{\delta(t)}\bigl(
			\eexp{\eta(t)}h(\bsx+\eexp{-\alpha(t)}\bsv) - \eexp{\eta(t)}h(\bsx) - \eexp{\eta(t)-\alpha(t)}\nabla^\top h(\bsx)\bsv
			-f(\bsx)
			\bigr).
		\end{split}
	\end{equation}
	The partial derivatives in the Euler--Lagrange equation are given as follows:
	\begin{align}
		\pdif{\mathcal{L}(\bsx(t),\dot{\bsx}(t),t)}{\bsx}
		                                                  & = \eexp{\delta(t)}\bigl(\eexp{\eta(t)}[\nabla h(\bsz(t))-\nabla h(\bsx(t))] - \eexp{\eta(t)-\alpha(t)}\nabla^2 h(\bsx(t))\dot{\bsx}(t) - \nabla f(\bsx(t))\bigr)                       \\
		                                                  & = \eexp{\delta(t)+\eta(t)-\alpha(t)}\bigl(\eexp{\alpha(t)}[\nabla h(\bsz(t))-\nabla h(\bsx(t))] - \nabla^2 h(\bsx(t))\dot{\bsx}(t) - \eexp{-\eta(t)+\alpha(t)}\nabla f(\bsx(t))\bigr), \\
		\pdif{\mathcal{L}(\bsx(t),\dot{\bsx}(t),t)}{\bsv} & = \eexp{\delta(t)+\eta(t)-\alpha(t)}\left[\nabla h(\bsz(t))- \nabla h(\bsx(t))\right],
	\end{align}
	and its time derivative is
	\begin{equation}
		\begin{split}
			 & \dif{}{t}\pdif{\mathcal{L}(\bsx(t),\dot{\bsx}(t),t)}{\bsv}                                                                                    \\
			 & = \eexp{\delta(t)+\eta(t)-\alpha(t)}\bigl((\dot{\delta}(t) +\dot{\eta}(t) - \dot{\alpha}(t))\left[\nabla h(\bsz(t))- \nabla h(\bsx(t))\right]
			+ \nabla^2 h(\bsz(t))\dot{\bsz}(t) - \nabla^2 h(\bsx(t))\dot{\bsx}(t)\bigr).
		\end{split}
		\elab{ddtdLdv}
	\end{equation}
	A term $-\eexp{\delta(t)+\eta(t) -\alpha(t)}\nabla^2 h(\bsx(t))\dot{\bsx}(t)$, which appears in both the side of the Euler--Lagrange equation, cancels out,
	and the substitution of the partial derivatives above into the Euler--Lagrange equation \eref{EL} results in the second equality in Eq.~\eref{stateEq} of the theorem.
	Furthermore,
	on using a time derivative of $\bsz(t)$,
	which is $\dot{\bsz}(t) = (1-\dot{\alpha}(t)\eexp{-\alpha(t)})\dot{\bsx}(t)+ \eexp{-\alpha(t)}\ddot{\bsx}(t)$,
	the second equation of Eq.~\eref{stateEq} results in Eq.~\eref{ODE}.
\end{proof}

\begin{remark}\remarklab{LagrangianScaling}
	The scaling with a constant in the Lagrangian does not affect the resulting ODEs, derived from the Euler--Lagrange equation.
	Therefore, the variation over time of $\eexp{\delta(t)}$ is essential while
	the initial value of $\delta(t)$ (i.e., $\delta(t_0)$) also does not affect the resulting ODEs.
\end{remark}
\begin{remark}
	Each time-varying parameters (e.g., $\alpha(t),\delta(t)$, and $\eta(t)$ in the theorem above) needs to be defined on $[t_0,T]$.
	The case of the dumping term $(C/t)\dot{\bsx}(t)$ is considered as an example;
	the dumping coefficient $C/t$ is not defined at $t = 0$, and the beginning of the time interval is set as $t_0 > 0$.
\end{remark}

\subsection{Lyapunov Function}
This section presents a Lyapunov function, which plays an important role in the convergence analysis,
for the proposed Euler--Lagrange system.
The subsequent proof focuses on an insight that
the fundamental of the proofs in \cite{Wi21,Ki23} comprises
decomposition technique using the Bregman three-point identity
and an assumption that $f$ is $\sigma$-uniformly convex with respect to $h$.
That is, the refinement of the conditions on the time-varying parameters is performed
using the newly introduced parameters $\delta(t), \eta(t)$, and $\nu(t)$.

First, the assumption on the differentiable time-varying parameters $\alpha(t), \delta(t), \eta(t), \nu(t)$ on $[t_0,T]$
is placed as follows:
\begin{assumption}\assumptionlab{general}
	\begin{equation}
		\begin{cases}
			\dot{\nu}(t) \leq \eexp{\alpha(t)},                                                                             \\
			-[\dot{\delta}(t) + \dot{\eta}(t) - \dot{\alpha}(t) - \eexp{\alpha(t)}]
			+ [\dot{\nu}(t) + \dot{\eta}(t)] \leq 0,                                                                        \\
			[\dot{\delta}(t) + \dot{\eta}(t) - \dot{\alpha}(t) - \eexp{\alpha(t)}] -\sigma\eexp{-\eta(t)+\alpha(t)} \leq 0, \\
			-[\dot{\delta}(t) + \dot{\eta}(t) - \dot{\alpha}(t)-\eexp{\alpha(t)}] \leq 0.
		\end{cases}
		\elab{general}
	\end{equation}
\end{assumption}
The following theorem provides a Lyapunov function, which results in the convergence rate formulae
in subsequent \theoremref{LyapunovBounds}.
\begin{theorem}\theoremlab{Lyapunov}
	Under Assumptions \ref{assumption:0} and \ref{assumption:general},
	the following function $V(t)$ of $\bsx(t)$ and $\bsz(t) = \bsx(t) + \eexp{-\alpha(t)}\dot{\bsx}(t)$ in the state equation \eref{stateEq}
	is nonincreasing on $[t_0,T]$, which means $\dot{V}(t) \leq  0\,(t\in[t_0,T])$:
	\begin{equation}
		V(t) = \eexp{\nu(t)}\bigl( \eexp{\eta(t)}\mathcal{D}_h(\bsx^*,\bsz(t)) + [f(\bsx(t)) - f(\bsx^*)] \bigr).
		\elab{Lyapunov}
	\end{equation}
\end{theorem}
\begin{proof}
	On recalling the time derivative of the Bregman divergence \eref{bregdivdiff},
	that of $\mathcal{D}_h(\bsx^*,\bsz(t)) = h(\bsx^*) - h(\bsz(t)) - \nabla^\top h(\bsz(t))[\bsx^*-\bsz(t)]$ is calculated as follows:
	\begin{equation}
		\begin{split}
			 & \dif{\mathcal{D}_h(\bsx^*,\bsz(t))}{t}            = - [\nabla^2 h(\bsz(t))\dot{\bsz}(t)]^\top[\bsx^*-\bsz(t)]                                                            \\
			 & = (\dot{\delta}(t) + \dot{\eta}(t) - \dot{\alpha}(t) -\eexp{\alpha(t)})[\nabla h(\bsz(t)) - \nabla h(\bsx(t))]^\top[\bsx^* - \bsz(t)]
			+ \eexp{-\eta(t)+\alpha(t)}\nabla^\top f(\bsx(t))[\bsx^* - \bsz(t)]                                                                                                         \\
			 & = (\dot{\delta}(t) + \dot{\eta}(t) - \dot{\alpha}(t) -\eexp{\alpha(t)})[-\mathcal{D}_h(\bsx^*,\bsz(t)) + \mathcal{D}_h(\bsx^*,\bsx(t)) - \mathcal{D}_h(\bsz(t),\bsx(t))] \\
			 & \quad + \eexp{-\eta(t)+\alpha(t)}\nabla^\top f(\bsx(t))[\bsx^* - \bsz(t)].
		\end{split}
		\elab{dDhdt}
	\end{equation}
	The second equality uses Eq.~\eref{stateEq} (the expression of $\nabla^2 h(\bsz(t))\dot{\bsz}(t)$),
	and the last equality uses the Bregman three-point identity \eref{threepoint} for
	$(\bsx_1,\bsx_2,\bsx_3) = (\bsx^*,\bsz, \bsx)$.

	By using $\dot{\bsx}(t) = \eexp{\alpha(t)}(\bsz(t) - \bsx(t))$, the time derivative of $f(\bsx(t))$ is given as follows:
	\begin{equation}
		\dif{f(\bsx(t))}{t} = \nabla^\top f(\bsx(t))\dot{\bsx}(t) = \eexp{\alpha(t)}\nabla^\top f(\bsx(t))[\bsz(t)-\bsx(t)].
		\elab{dfxdt}
	\end{equation}
	The formulae of the time derivatives above result in the following time derivative of the Lyapunov function candidate $V(t)$:
	\begin{equation}
		\begin{split}
			\dot{V}(t) & = [\dot{\nu}(t) + \dot{\eta}(t)]\eexp{\nu(t)+\eta(t)}\mathcal{D}_h(\bsx^*,\bsz(t)) + \eexp{\nu(t)+\eta(t)}\dif{\mathcal{D}_h(\bsx^*,\bsz(t))}{t} \\
			           & + \dot{\nu}(t)\eexp{\nu(t)}[f(\bsx(t)) - f(\bsx^*)] + \eexp{\nu(t)}\dif{f(\bsx(t))}{t}.
		\end{split}
		\elab{Vdot}
	\end{equation}
	The following terms are taken into consideration here:
	\begin{enumerate}
		\item the last term $\eexp{-\eta(t)+\alpha(t)}\nabla^\top f(\bsx(t))[\bsx^* - \bsz(t)]$ in $\tdif{\mathcal{D}_h(\bsx^*,\bsz(t))}{t}$ of Eq.~\eref{dDhdt},
		\item the third and forth terms $\dot{\nu}(t)\eexp{\nu(t)}[f(\bsx(t)) - f(\bsx^*)] + \eexp{\nu(t)}\tdif{f(\bsx(t))}{t}$ in $\dot{V}(t)$ of Eq.~\eref{Vdot}.
	\end{enumerate}
	Inequalities $\dot{\nu}(t) \leq \eexp{\alpha(t)}$ (i.e., the first item of Eq.~\eref{general} in \assumptionref{general})
	and $\mathcal{D}_{f}(\bsx^*,\bsx(t)) \geq \sigma\mathcal{D}_{h}(\bsx^*,\bsx(t))$ (i.e., the $\sigma$-uniform convexity of $h$ with respect to $h$)
	result in the following inequality:
	\begin{equation}
		\begin{split}
			 & \eexp{\nu(t)+\alpha(t)}\nabla^\top f(\bsx(t))[\bsx^* - \bsz(t)]
			+ \dot{\nu}(t)\eexp{\nu(t)}[f(\bsx(t)) - f(\bsx^*)] + \eexp{\nu(t)+\alpha(t)}\nabla^\top f(\bsx(t))[\bsz(t)-\bsx(t)]             \\
			 & \leq \eexp{\nu(t)+\alpha(t)}\left(\nabla^\top f(\bsx(t))[\bsx^* - \bsx(t)] + [f(\bsx(t)) - f(\bsx^*)]\right)                  \\
			 & = -\eexp{\nu(t)+\alpha(t)}\mathcal{D}_{f}(\bsx^*,\bsx(t)) \leq -\sigma\eexp{\nu(t)+\alpha(t)}\mathcal{D}_{h}(\bsx^*,\bsx(t)).
		\end{split}
	\end{equation}
	Consequently, $\dot{V}(t)$ is upper-bounded using $\mathcal{D}_{f}$ and $\mathcal{D}_h$ in $\tdif{\mathcal{D}_h(\bsx^*,\bsz(t))}{t}$ (Eq.~\eref{dDhdt}):
	\begin{equation}
		\begin{split}
			\dot{V}(t) & \leq \eexp{\nu(t)+\eta(t)}\bigl(-[\dot{\delta}(t) + \dot{\eta}(t) - \dot{\alpha}(t) - \eexp{\alpha(t)}]
			+ [\dot{\nu}(t) + \dot{\eta}(t)]\bigr)\mathcal{D}_h(\bsx^*,\bsz(t))                                                                                                                  \\
			           & + \eexp{\nu(t)+\eta(t)}\bigl([\dot{\delta}(t) + \dot{\eta}(t) - \dot{\alpha}(t) - \eexp{\alpha(t)}] -\sigma\eexp{-\eta(t)+\alpha(t)}\bigr)\mathcal{D}_h(\bsx^*,\bsx(t)) \\
			           & +\eexp{\nu(t)+\eta(t)}\bigl(-[\dot{\delta}(t) + \dot{\eta}(t) - \dot{\alpha}(t)-\eexp{\alpha(t)}]\bigr)\mathcal{D}_h(\bsz(t),\bsx(t)).
		\end{split}
		\elab{Ds}
	\end{equation}
	The Bregman divergence is always nonnegative,
	and the coefficient of each Bregman divergence is nonpositive because of the second, third, and forth items of the inequalities \eref{general} in \assumptionref{general};
	therefore, the inequality $\dot{V}(t) \leq 0$ holds.
\end{proof}
\begin{remark}\remarklab{LyapunovScaling}
	Similar to the Lagrangian, the same proof holds for a scaled Lyapunov function with a constant.
	Therefore, the initial value of $\nu(t)$ (i.e., $\nu(t_0)$) is arbitrary.
\end{remark}
By using the inequality \eref{Ds}, the following inequalities are obtained.
The first and second items show the convergence rates of the objective function value and
the Bregman divergence, respectively.
The third, forth, and fifth items are so-called integral estimates, which have been examined in existing studies
(e.g., \cite{At19,Qu22})
\begin{theorem}\theoremlab{LyapunovBounds}
	Under the same setting as \theoremref{Lyapunov}, the following inequalities holds at $t \in[ t_0,T]$:
	\begin{equation}
		\begin{dcases}
			f(\bsx(t)) - f(\bsx^*) \leq \eexp{-\nu(t)}V(t_0),                      \\
			\eexp{\eta(t)}\mathcal{D}_h(\bsx^*,\bsz(t)) \leq \eexp{-\nu(t)}V(t_0), \\
			\int_{t_0}^{+\infty}
			\eexp{\nu(t)+\eta(t)}\left([\dot{\delta}(t) + \dot{\eta}(t) - \dot{\alpha}(t) - \eexp{\alpha(t)}]
			- [\dot{\nu}(t) + \dot{\eta}(t)]\right)\mathcal{D}_h(\bsx^*,\bsz(t))\,\mathrm{d}t
			\leq V(t_0),                                                           \\
			\int_{t_0}^{+\infty}
			\eexp{\nu(t)+\eta(t)}\left(-[\dot{\delta}(t) + \dot{\eta}(t) - \dot{\alpha}(t) - \eexp{\alpha(t)}] + \sigma\eexp{-\eta(t)+\alpha(t)}\right)\mathcal{D}_h(\bsx^*,\bsx(t))\,
			\mathrm{d}t \leq V(t_0),                                               \\
			\int_{t_0}^{+\infty} \eexp{\nu(t)+\eta(t)}[\dot{\delta}(t) + \dot{\eta}(t) - \dot{\alpha}(t)-\eexp{\alpha(t)}]\mathcal{D}_h(\bsz(t),\bsx(t))\,\mathrm{d}t \leq V(t_0).
		\end{dcases}
	\end{equation}
\end{theorem}
\begin{proof}
	The definition of the Lyapunov function $V(t)$ in Eq.~\eref{Lyapunov} and
	and an inequality $\dot{V}(t) \leq 0$ are utilized here.
	Inequalities $f(\bsx(t)) - f(\bsx^*) \leq 0$ and $\bregdiv{h}{\bsx(t)}{\bsx^*} \leq 0$ result in
	\begin{equation}
		\begin{dcases}
			\eexp{\nu(t)}(f(\bsx(t)) - f(\bsx^*)) \leq V(t) \leq V(t_0), \\
			\eexp{\nu(t)+\eta(t)}\bregdiv{h}{\bsx^*}{\bsz(t)} \leq V(t) \leq V(t_0).
		\end{dcases}
	\end{equation}
	The rearrangement of the inequality above results in the first and second items of the inequalities that the theorem claims.

	Next, the integrals of both the sides of the inequality \eref{Ds} is taken into consideration.
	The fifth item, addressing $\mathcal{D}_h(\bsz(t),\bsx(t))$, is considered here as the representative of
	the third, forth, and fifth items. An inequality
	\begin{equation}
		-V(t_0) \leq V(t) - V(t_0) \leq \int_{t_0}^{t} \eexp{\nu(t)+\eta(t)}(-[\dot{\delta}(t) + \dot{\eta}(t) - \dot{\alpha}(t)-\eexp{\alpha(t)}])\mathcal{D}_h(\bsz(t),\bsx(t))
		\,\mathrm{d}t
		\leq 0
	\end{equation}
	follows, and the rearrangement and taking the limit of $t\to+\infty$ result in the inequality.
	The third and forth items are similarly obtained.
\end{proof}

\subsection{Comparison with Lagrangian in Existing Study \cite{Ki23}}
The reports \cite{Wi16,Wi21} introduced two Lagrangians, called first and second Lagrangians,
separately for non-strongly convex and strongly convex objective functions.
To overcome this drawback, the report \cite{Ki23} proposed the following Lagrangian
using a new parameter $\gamma(t)$ satisfying $\dot{\gamma}(t) = \eexp{\alpha(t)}$:
\begin{equation}
	\begin{split}
		\mathcal{L}(\bsx,\bsv,t) & = \eexp{\alpha(t)+\gamma(t)}
		\bigl((1+\sigma\eexp{\beta(t)})\mathcal{D}_h(\bsx+\eexp{-\alpha(t)}\bsv,\bsx) - \eexp{\beta(t)}f(\bsx)\bigr) \\
		                         & = \eexp{\alpha(t)+\gamma(t)+\beta(t)}
		\bigl((1+\sigma\eexp{\beta(t)})\eexp{-\beta(t)}\mathcal{D}_h(\bsx+\eexp{-\alpha(t)}\bsv,\bsx) - f(\bsx)\bigr).
	\end{split}
	\elab{ExistingLagrangian2}
\end{equation}
By using the auxiliary variable $\bsz(t) = \bsx(t) + \eexp{-\alpha(t)}\dot{\bsx}(t)$, a Lyapunov function is given by
\begin{equation}
	\begin{split}
		V(t) & = (1+\sigma\eexp{\beta(t)})\mathcal{D}_h(\bsx^*,\bsz(t)) +\eexp{\beta(t)}(f(\bsx(t)) - f(\bsx^*))                                \\
		     & = \eexp{\beta(t)}\bigl( (1+\sigma\eexp{\beta(t)})\eexp{-\beta(t)}\mathcal{D}_h(\bsx^*,\bsz(t)) + [f(\bsx(t)) - f(\bsx^*)]\bigr).
	\end{split}
	\elab{ExistingLyapunovFunc2}
\end{equation}
$\beta(t)$ is chosen so that $\dot{\beta}(t) \leq \eexp{\alpha(t)}$;
actually, it is set as the upper bound $\dot{\gamma}(t) = \dot{\beta}(t) = \eexp{\alpha(t)}$
to improve the convergence rate. The Lagrangian above respects
the first and second Euler--Lagrange equations \cite{Wi16,Wi21}
and is introduced to combine both the equations (see, subsection D.1. of \cite{Ki23}) and still mysterious.
The setting is summarized as follows:
\begin{assumption}\assumptionlab{Kim}
	\begin{equation}
		\begin{dcases}
			h(\bsx) = \frac{1}{2}\|\bsx\|_2^2,           \\
			0 \leq \dot{\beta}(t) \leq \eexp{\alpha(t)}, \\
			\dot{\gamma}(t) = \eexp{\alpha(t)}.
		\end{dcases}
	\end{equation}
\end{assumption}
\begin{remark}
	The first inequality $0 \leq \dot{\beta}(t)$ of the second item is not explicitly shown in \cite{Wi21,Ki23}
	while ``ideal scaling conditions''\cite{Wi21} assumes only the second inequality $\dot{\beta}(t) \leq \eexp{\alpha(t)}$.
	However, both subsection A.2 of \cite{Wi21} and subsection C.3. of \cite{Ki23} use
	that $\sigma \dot{\beta}(t)\eexp{\beta(t)}$ is nonnegative, and hence so is $\dot{\beta}(t)$.
\end{remark}
In this paper, both the Lagrangian and Lyapunov function are generalized;
under the following setting, the proposed model of this study is consistent with that of \cite{Ki23} (see, subsection 3.1. and D.1. of \cite{Ki23} for further details).
That is, the substitution of the following parameters into the Lagrangian \eref{Lagrangian} and Lyapunov function \eref{Lyapunov}
results in the Lagrangian of \cite{Ki23} (Eq.~(4) of \cite{Ki23}, Eq.~\eref{ExistingLagrangian2} of this paper)
and Lyapunov function of \cite{Ki23} (Eq.~(63) of \cite{Ki23}, Eq.~\eref{ExistingLyapunovFunc2} of this paper):
\begin{equation}
	\begin{dcases}
		\dot{\nu}(t) = \dot{\beta}(t),                                                                                                      \\
		\eta(t) = - \beta(t) + \log(1+\sigma\eexp{\beta(t)})\,\,\text{i.e.,}\,\,\eexp{\eta(t)} = (1+\sigma\eexp{\beta(t)})\eexp{-\beta(t)}, \\
		\dot{\delta}(t) = \dot{\alpha}(t) + \dot{\gamma}(t) + \dot{\beta}(t) = \dot{\alpha}(t) + \eexp{\alpha(t)} + \dot{\beta}(t).
	\end{dcases}
	\elab{ExistingParameters}
\end{equation}
While \cite{Ki23} shows that the existing Lyapunov function \eref{ExistingLyapunovFunc2} works
(i.e., an inequality $\dot{V}(t) \leq 0$ holds) in the setting above,
this paper shows the convergence analysis using the proposed Lyapunov function \eref{Lyapunov}.
In other words, \assumptionref{general} is examined under the setting above;
if \assumptionref{general} is satisfied, \theoremref{Lyapunov} claims that the Lyapunov function \eref{ExistingLyapunovFunc2} works.
The following calculation follows:
\begin{equation}
	\dot{\eta}(t) = -\dot{\beta}(t) + \sigma\dot{\beta}(t)\eexp{\beta(t)}/(1+\sigma\eexp{\beta(t)}),
\end{equation}
\begin{equation}
	\dot{\delta}(t) + \dot{\eta}(t) - \dot{\alpha}(t)-\eexp{\alpha(t)} = \sigma\dot{\beta}(t)\eexp{\beta(t)}/(1+\sigma\eexp{\beta(t)}).
\end{equation}
By using the results above,
the left sides of the inequalities in \assumptionref{general} are calculated as follows:
\begin{equation}
	\begin{cases}
		\dot{\nu}(t) = \dot{\beta}(t) \leq \eexp{\alpha(t)},                                         \\
		-[\dot{\delta}(t) + \dot{\eta}(t) - \dot{\alpha}(t) - \eexp{\alpha(t)}]
		+ [\dot{\nu}(t) + \dot{\eta}(t)]
		= 0,                                                                                         \\
		[\dot{\delta}(t) + \dot{\eta}(t) - \dot{\alpha}(t) - \eexp{\alpha(t)}] -\sigma\eexp{-\eta(t)+\alpha(t)}
		= (\dot{\beta}(t) - \eexp{\alpha(t)})\sigma\eexp{\beta(t)}/(1+\sigma\eexp{\beta(t)}) \leq 0, \\
		-[\dot{\delta}(t) + \dot{\eta}(t) - \dot{\alpha}(t)-\eexp{\alpha(t)}]
		= -\sigma\dot{\beta}(t)\eexp{\beta(t)}/(1+\sigma\eexp{\beta(t)}) \leq 0.
	\end{cases}
\end{equation}
\assumptionref{general} is satisfied; therefore, the Lyapunov function \eref{ExistingLyapunovFunc2} works.
\begin{remark}\remarklab{Ki23}
	The substitution of the parameters \eref{ExistingParameters} into the ODE \eref{ODE}
	results in the following equation \cite{Ki23}
	(equivalently, the substitution of the parameters \eref{ExistingParameters} into the Lagrangian \eref{Lagrangian} results in the same ODE):
	\begin{equation}
		\ddot{\bsx}(t) +
		\left(\eexp{\alpha(t)} - \dot{\alpha}(t) + \frac{\sigma\dot{\beta}(t)\eexp{\beta(t)}}{1+\sigma\eexp{\beta(t)}}\right)\dot{\bsx}(t)
		+\frac{\eexp{2\alpha(t)+\beta(t)}}{1+\sigma\eexp{\beta(t)}}\nabla f(\bsx(t)) = \bszero_{n_\bsx}.
		\elab{ExistingODE}
	\end{equation}
	The coefficient of $\nabla f(\bsx(t))$, given by $\eexp{2\alpha(t)+\beta(t)}/(1+\sigma\eexp{\beta(t)})$,
	is not necessarily $1$ and not equal to that of $\ddot{\bsx}(t)$.
	Particularly, the subsequent models (e.g., a constant damping coefficient $D$ and a dumping coefficient $C/t$)
	are difficult to be covered by the model \eref{ExistingODE}.
	Actually, the case of a dumping term $D\dot{\bsx}(t)$ is asymptotically realized by the model \eref{ExistingODE} (see, subsections C.4, C.5 and D.1 in \cite{Ki23}).

	Furthermore, the subsequent discussion on the case of $h(\bsx) = (1/2)\|\bsx\|_2^2$ and the resulting ODE \eref{ODE}
	indicates that a condition $\eta(t) = 2\alpha(t)$ is necessary to match the coefficients of $\ddot{\bsx}(t)$ and $\nabla f(\bsx(t))$.
	The resulting ODE has a simple form:
	\begin{equation}
		\ddot{\bsx}(t) + \dot{\delta}(t)\dot{\bsx}(t) + \nabla f(\bsx(t)) = \bszero_{n_\bsx}.
	\end{equation}
	The subsequent actual models are the variants of the aforementioned ODE. In addition,
	in the case of $\eta(t) = 2\alpha(t)$, the Taylor expansion is applied to $h$ as follows:
	\begin{equation}
		h(\bsx+\eexp{-\alpha(t)}\bsv) = h(\bsx) + \eexp{-\alpha(t)}\nabla^\top h(\bsx)\bsv + \frac{\eexp{-2\alpha(t)}}{2}\bsv^\top\nabla^2 h(\bsx)\bsv
		+ \mathcal{O}(\|\eexp{-\alpha(t)}\bsv\|_2^3).
	\end{equation}
	Therefore, the Bregman divergence in the Lagrangian is approximated as
	\begin{equation}
		\eexp{2\alpha(t)}\mathcal{D}_h(\bsx+\eexp{-\alpha(t)}\bsv,\bsx) \approx \frac{1}{2}\bsv^\top\nabla^2 h(\bsx)\bsv,
	\end{equation}
	and it seems the quadratic form without the effect of $\eexp{\alpha(t)}$.
	Particularly, if $h(\bsx) = (1/2)\|\bsx\|_2^2$, the approximation above results in the equality (i.e., $\eexp{2\alpha(t)}\mathcal{D}_h(\bsx+\eexp{-\alpha(t)}\bsv,\bsx) = (1/2)\|\bsv\|_2^2$).
	This means that the resulting Lagrangian
	$\mathcal{L}(\bsx,\bsv,t) = \eexp{\delta(t)}\left((1/2)\|\bsv\|_2^2 - f(\bsx)\right)$ represents
	the scaled value of that for the equation of motion of a unit mass under the potential $f(\bsx)$,
	given by $(1/2)\|\bsv\|_2^2 - f(\bsx)$.
\end{remark}

\section{Case of Symmetric Bregman Divergence}\sectionlab{Symmetric}
As reported in \cite{At19}, the Bregman divergence associated with $h(\bsx) = (1/2)\|\bsx\|_2^2$
can include a term $(1/2)\|\bsx^*-\bsx(t)\|_2^2$ in a Lyapunov function
and have a potential to improve the convergence rate formula.
This section presents a generalization of the aforementioned technique and
shows that the symmetric Bregman divergence, which is a broader class of the $\ell_2$-norm example,
can add a new term $\bregdiv{h}{\bsx(t)}{\bsx^*}$ in the Lyapunov function proposed in the previous section.
A special case of $h(\bsx) = \frac{1}{2}\|\bsx\|_2^2$ and $(C/t)\dot{\bsx}(t)$ is consistent with \cite{At19}.
The differences between this section and \cite{At19} are summarized as follows:
\begin{itemize}
	\item This paper utilizes the auxiliary variable $\bsz(t) = \bsx(t) + \eexp{-\alpha(t)}\dot{\bsx}(t)$,
	      which is simultaneously designed with the Lyapunov function using the parameter $\eexp{-\alpha(t)}$.
	\item The derivation of the parameters is somewhat natural in this paper.
	      Precisely, the conditions of the parameters are naturally obtained to
	      satisfy $\dot{V}(t) \leq 0$.
\end{itemize}

\begin{assumption}\assumptionlab{general2}
	\begin{equation}
		\begin{cases}
			\forany{\bsx}, \bsy, \,\, \bregdiv{h}{\bsx}{\bsy} = \bregdiv{h}{\bsy}{\bsx},              \\
			\dot{\nu}(t) \leq \eexp{\alpha(t)},                                                       \\
			-[\dot{\delta}(t) + \dot{\eta}(t) - \dot{\alpha}(t) - \eexp{\alpha(t)}]
			+ [\dot{\nu}(t) + \dot{\eta}(t)] + \eexp{\pi(t)+\alpha(t)} \leq 0,                        \\
			[\dot{\delta}(t) + \dot{\eta}(t) - \dot{\alpha}(t) - \eexp{\alpha(t)}] -\sigma\eexp{-\eta(t)+\alpha(t)}
			-\eexp{\pi(t)+\alpha(t)} + [\dot{\nu}(t)+\dot{\eta}(t)+\dot{\pi}(t)]\eexp{\pi(t)} \leq 0, \\
			-[\dot{\delta}(t) + \dot{\eta}(t) - \dot{\alpha}(t)-\eexp{\alpha(t)}] - \eexp{\pi(t)+\alpha(t)} \leq 0.
		\end{cases}
		\elab{general2}
	\end{equation}
\end{assumption}
\begin{remark}
	\assumptionref{general} is a special case of \assumptionref{general2} by replacing $\eexp{\pi(t)}$ with zero (i.e., $\pi(t)\to -\infty\,\,(t\in[t_0,T])$).
	That is, the assumption of the symmetric Bregman divergence relaxes \assumptionref{general}.
	The parameter $\pi(t)$ is exploited in the case of the time-varying dumping coefficient $(C/t)\dot{\bsx}(t)$ in
	subsequent \subsectionref{C/t}.
\end{remark}
The following theorem claims that a term $\bregdiv{h}{\bsx(t)}{\bsx^*}$ can be included in the Lyapunov function given in \theoremref{Lyapunov} under the relaxed conditions in \assumptionref{general2}.
\begin{theorem}\theoremlab{Lyapunov2}
	Under \assumptionref{0} and \assumptionref{general2} for $\alpha(t), \delta(t), \eta(t), \nu(t)$, and $\pi(t)$,
	a function $V(t) = \eexp{\nu(t)}\bigl( \eexp{\eta(t)}[\bregdiv{h}{\bsx^*}{\bsz(t)} + \eexp{\pi(t)}\bregdiv{h}{\bsx(t)}{\bsx^*}]
		+ [f(\bsx(t)) - f(\bsx^*)] \bigr)$ with $\bsx(t)$ and $\bsz(t) = \bsx(t) + \eexp{-\alpha(t)}\dot{\bsx}(t)$ in the state equation \eref{stateEq} is nonincreasing on $[t_0,T]$, which means $\dot{V}(t) \leq  0\,(t\in[t_0,T])$.
\end{theorem}
\begin{proof}
	The time derivative of $\bregdiv{h}{\bsx(t)}{\bsx^*}$ is calculated as follows:
	\begin{equation}
		\begin{split}
			\dif{\bregdiv{h}{\bsx(t)}{\bsx^*}}{t}
			 & = \dot{\bsx}^\top(t)[\nabla h(\bsx(t)) - \nabla h(\bsx^*)]  =
			\eexp{\alpha(t)}[\bsz(t)- \bsx(t)]^\top[\nabla h(\bsx(t)) - \nabla h(\bsx^*)]                                                      \\
			 & = \eexp{\alpha(t)}\left[ - \bregdiv{h}{\bsz(t)}{\bsx(t)} + \bregdiv{h}{\bsz(t)}{\bsx^*} - \bregdiv{h}{\bsx(t)}{\bsx^*} \right],
		\end{split}
	\end{equation}
	where the third equality follows form the Bregman three-point identity \eref{threepoint} for $(\bsx_1,\bsx_2,\bsx_3) = (\bsz(t),\bsx(t), \bsx^*)$.
	The time derivative of $V(t)$ is given as follows:
	\begin{equation}
		\begin{split}
			\dot{V}(t) & = [\dot{\nu}(t) + \dot{\eta}(t)]\eexp{\nu(t)+\eta(t)}\bregdiv{h}{\bsx^*}{\bsz(t)}
			+ \eexp{\nu(t)+\eta(t)}\dif{\bregdiv{h}{\bsx^*}{\bsz(t)}}{t}                                                          \\
			           & + [\dot{\nu}(t) + \dot{\eta}(t) + \dot{\pi}(t)] \eexp{\nu(t)+\eta(t)+\pi(t)}\bregdiv{h}{\bsx(t)}{\bsx^*}
			+ \eexp{\nu(t)+\eta(t)+\pi(t)}\dif{\bregdiv{h}{\bsx(t)}{\bsx^*}}{t}                                                   \\
			           & + \dot{\nu}(t)\eexp{\nu(t)}[f(\bsx(t)) - f(\bsx^*)] + \eexp{\nu(t)}\dif{f(\bsx(t))}{t}.
		\end{split}
	\end{equation}
	On taking into consideration the effect of terms having $\eexp{\pi(t)}$ above,
	the remaining part of this proof evaluates $\dot{V}(t)$ by using Eq.~\eref{Vdot} in the proof of \theoremref{Lyapunov}.
	The inequality \eref{Ds} is revised by taking into consideration $\eexp{\pi(t)}$,
	and the symmetry of the Bregman divergence terms in $\tdif{\bregdiv{h}{\bsx(t)}{\bsx^*}}{t}$
	results in the following upper bound of $\dot{V}(t)$:
	\begin{equation}
		\begin{split}
			 & \dot{V}(t)                                                                                                                                                       \\
			 & \leq \eexp{\nu(t)+\eta(t)}\bigl(-[\dot{\delta}(t) + \dot{\eta}(t) - \dot{\alpha}(t) - \eexp{\alpha(t)}]
			+ [\dot{\nu}(t) + \dot{\eta}(t)] + \eexp{\pi(t)+\alpha(t)}\bigr)\bregdiv{h}{\bsx^*}{\bsz(t)}                                                                        \\
			 & + \eexp{\nu(t)+\eta(t)}\bigl([\dot{\delta}(t) + \dot{\eta}(t) - \dot{\alpha}(t) - \eexp{\alpha(t)}] -\sigma\eexp{-\eta(t)+\alpha(t)}
			-\eexp{\pi(t)+\alpha(t)} + [\dot{\nu}(t)+\dot{\eta}(t)+\dot{\pi}(t)]\eexp{\pi(t)}\bigr)\bregdiv{h}{\bsx^*}{\bsx(t)}                                                 \\
			 & + \eexp{\nu(t)+\eta(t)}\bigl(-[\dot{\delta}(t) + \dot{\eta}(t) - \dot{\alpha}(t)-\eexp{\alpha(t)}] - \eexp{\pi(t)+\alpha(t)}\bigr)\bregdiv{h}{\bsz(t)}{\bsx(t)}.
		\end{split}
		\elab{Ds2}
	\end{equation}
	\assumptionref{general2} indicates that the coefficient of each Bregman divergence is nonpositive,
	the inequality $\dot{V}(t) \leq 0$ is obtained.
\end{proof}
Similar to \theoremref{LyapunovBounds}, obtained using \theoremref{Lyapunov},
\theoremref{Lyapunov2} provides upper bounds of
$f(\bsx(t)) - f(\bsx^*), \mathcal{D}_h(\bsx^*,\bsz(t)),
	\mathcal{D}_h(\bsx^*,\bsx(t))$, and $\mathcal{D}_h(\bsz(t),\bsx(t))$;
resulting inequalities are omitted here due to space limitation.

\subsection{Case of $\ell_2$-Norm}\subsectionlab{ell2}
The analysis in the previous section is somewhat overgeneralized
and has the various parameters ($\alpha(t), \delta(t), \eta(t)$, and $\nu(t)$) to be designed;
it is inconvenient in practical algorithm design.
This subsection presents the standard form covering a broader class of ODE models.
The following setting is taken into consideration here.
\begin{itemize}
	\item In the ODE \eref{ODE}, obtained from the Euler--Lagrange equation,
	      the conventional models have implicitly assumes that the coefficients of $\ddot{\bsx}(t)$ and $\nabla f(\bsx(t))$ are the same;
	      thus $\eta(t) = 2\alpha(t)$ is chosen to satisfy the requirement.
	\item The Bregman divergence is introduced with $h(\bsx) = (1/2)\|\bsx\|_2^2$.
\end{itemize}

\begin{assumption}\assumptionlab{para}
	\begin{equation}
		\begin{cases}
			h(\bsx) = \dfrac{1}{2}\|\bsx\|_2^2,                                                                                        \\
			\eta(t) = 2\alpha(t),                                                                                                      \\
			\dot{\nu}(t) \leq \eexp{\alpha(t)},                                                                                        \\
			-[\dot{\delta}(t) + \dot{\alpha}(t) - \eexp{\alpha(t)}] +[\dot{\nu}(t)+2\dot{\alpha}(t)] + \eexp{\pi(t)+\alpha(t)} \leq 0, \\
			[\dot{\delta}(t) + \dot{\alpha}(t) - \eexp{\alpha(t)}] -\sigma\eexp{-\alpha(t)}
			+ [\dot{\nu}(t) + 2\dot{\alpha}(t) + \dot{\pi}(t) - \eexp{\alpha(t)}]\eexp{\pi(t)} \leq 0,                                 \\
			-[\dot{\delta}(t) + \dot{\alpha}(t) - \eexp{\alpha(t)}] - \eexp{\pi(t)+\alpha(t)} \leq 0.
		\end{cases}
	\end{equation}
\end{assumption}
Under \assumptionref{para},
the coefficient of $\dot{\bsx}(t)$ in Eq.~\eref{ODE} results in $\dot{\delta}(t)$,
and the ODE \eref{ODE} is reduced to an ODE
\begin{equation}
	\ddot{\bsx}(t) + \dot{\delta}(t)\dot{\bsx}(t) + \nabla f(\bsx(t)) = \bszero_{n_\bsx}
	\elab{ODEStandard}
\end{equation}
and an equivalent state equation
\begin{equation}
	\begin{dcases}
		\dot{\bsx}(t) = \eexp{\alpha(t)}(\bsz(t) - \bsx(t)), \\
		\dot{\bsz}(t) = -(\dot{\delta}+\dot{\alpha}(t)-\eexp{\alpha(t)})[\bsz(t) - \bsx(t)] - \eexp{-\alpha(t)}\nabla f(\bsx(t)).
	\end{dcases}
	\elab{ELStandard}
\end{equation}
Hereafter, the term $\dot{\delta}(t)\dot{\bsx}(t)$ in Eq.~\eref{ELStandard} and $\dot{\delta}(t)$ are called
the dumping term and dumping coefficient, respectively.

Similar to the proof of \theoremref{LyapunovBounds},
the integration of both the sides of the inequality \eref{Ds2} results in the following theorem.
\begin{theorem}\theoremlab{ell2LyapunovBounds}
	Under the same setting as \theoremref{Lyapunov} with \assumptionref{para},
	the following inequality holds for $t \in[ t_0,T]$:
	\begin{equation}
		\int_{t_0}^{+\infty}\eexp{\nu(t)}\left(
		[\dot{\delta}(t) + \dot{\alpha}(t) - \eexp{\alpha(t)}] + \eexp{\pi(t)+\alpha(t)}
		\right)\frac{1}{2}\|\dot{\bsx}(t)\|_2^2\,\mathrm{d}t \leq V(t_0).
	\end{equation}
\end{theorem}

\subsection{Case of Time-Invariant Dumping Coefficient $D\dot{\bsx}(t)$}\subsectionlab{TI}
To take into consideration the design direction of the parameters,
this subsection addresses the time-invariant dumping coefficient $D\dot{\bsx}(t)$ under the strong convexity ($\sigma>0$).
In advance, the designed parameters are summarized below:
\begin{equation}
	\begin{dcases}
		\alpha(t) = \begin{cases}
			            \log(D/2)                          & (0 < D \leq 2\sqrt{\sigma}),       \\
			            \log([D - \sqrt{D^2 - 4\sigma}]/2) & (2\sqrt{\sigma} \leq D < +\infty),
		            \end{cases} \\
		\dot{\delta}(t) = D,                                                                \\
		\eta(t) = 2\alpha(t),                                                               \\
		\dot{\nu}(t) = \eexp{\alpha(t)}                                                     \\
		\text{($\pi(t)$ is not used, i.e., $\pi(t)=-\infty$ or $\eexp{\pi(t)}=0$).}
	\end{dcases}
	\elab{TIParams}
\end{equation}
As mentioned in Remarks \ref{remark:LagrangianScaling} and \ref{remark:LyapunovScaling},
the time variation is significant in the design of  $\delta(t)$ and $\nu(t)$;
therefore, the derivatives of them are indicated as above. The resulting ODE is
\begin{equation}
	\ddot{\bsx}(t) + D\dot{\bsx}(t) + \nabla f(\bsx(t)) = \bszero_{n_\bsx}.
	\elab{StODE}
\end{equation}
For $D \leq 2\sqrt{\sigma}$, the state equation with the auxiliary variable $\bsz(t) = \bsx(t) + \dot{\bsx}(t)/(D/2)$ is given as follows:
\begin{equation}
	\begin{dcases}
		\dot{\bsx}(t) = (D/2)[\bsz(t) - \bsx(t)], \\
		\dot{\bsz}(t) = -(D/2)[\bsz(t) - \bsx(t)] - (2/D) \nabla f(\bsx(t)).
	\end{dcases}
	\elab{StODE1}
\end{equation}
Besides, in the case of $D \geq 2\sqrt{\sigma}$,
the auxiliary variable is defined by $\bsz(t) = \bsx(t) + \dot{\bsx}(t)/([D - \sqrt{D^2 - 4\sigma}]/2)$,
and the state equation is as follows:
\begin{equation}
	\begin{dcases}
		\dot{\bsx}(t) = \frac{D - \sqrt{D^2 - 4\sigma}}{2}[\bsz(t) - \bsx(t)], \\
		\dot{\bsz}(t) = -\frac{D + \sqrt{D^2 - 4\sigma}}{2}[\bsx(t) - \bsz(t)]
		- \frac{2}{D - \sqrt{D^2 - 4\sigma}}\nabla f(\bsx(t)).
	\end{dcases}
	\elab{StODE2}
\end{equation}
As \remarkref{Ki23} indicates, the coefficient of $\nabla f(\bsx(t))$ is not necessarily $1$ in the model in \cite{Ki23} and inapplicable.
It should be noted that this paper has already provided the generalized Lyapunov function in \theoremref{Lyapunov}
and does not need to give a proof for each model separately (see, \cite{Wi16,Wi21}, subsection C.5. of \cite{Ki23}).

The aforementioned parameters \eref{TIParams} are derived here.
$\nu(t)$ is preferred to diverge fast because the convergence rate of the objective function value is $\mathcal{O}(\eexp{-\nu(t)})$.
On recalling the upper bound $\dot{\nu}(t) \leq \eexp{\alpha(t)}$,
whether both $\dot{\nu}(t)$ and $\eexp{\alpha(t)}$ can diverge or not is examined.
In the fifth item $-[\dot{\delta}(t) + \dot{\alpha}(t) - \eexp{\alpha(t)}] +[\dot{\nu}(t)+2\dot{\alpha}(t)] + \eexp{\pi(t)+\alpha(t)} \leq 0$
and sixth item $[\dot{\delta}(t) + \dot{\alpha}(t) - \eexp{\alpha(t)}] -\sigma\eexp{-\alpha(t)}
	+ [\dot{\nu}(t) + 2\dot{\alpha}(t) + \dot{\pi}(t) - \eexp{\alpha(t)}]\eexp{\pi(t)} \leq 0$ in \assumptionref{para},
the arrangement with respect to $[\dot{\delta}(t) + \dot{\alpha}(t) - \eexp{\alpha(t)}]-\eexp{\alpha(t)+\pi(t)}$ results in
\begin{equation}
	[\dot{\nu}(t)+2\dot{\alpha}(t)] \leq \sigma\eexp{-\alpha(t)} -[\dot{\nu}(t) + 2\dot{\alpha}(t) + \dot{\pi}(t)]\eexp{\pi(t)}.
\end{equation}
By using $\mathrm{d}\eexp{\pi(t)}/\mathrm{d}t = \dot{\pi}(t)\eexp{\pi(t)}$, an inequality
\begin{equation}
	[\dot{\nu}(t)+2\dot{\alpha}(t)](1+\eexp{\pi(t)}) \leq \sigma\eexp{-\alpha(t)} -\dif{\eexp{\pi(t)}}{t}
\end{equation}
is obtained. $\tdif{\eexp{\pi(t)}}{t}$ does not diverge to $-\infty$ because of $\eexp{\pi(t)} > 0$.
Besides, if the left side diverges to $+\infty$, the inequality does not hold.
Therefore, $\dot{\nu}(t)$ diverging to $+\infty$ cannot be taken in this framework.

Based on the aforementioned discussion and the study on the steady state behavior of the hyperbolic-function-based parameters (subsection C.4. of \cite{Ki23}),
the steady state situation of the parameters, which means that the derivative is zero and the parameter is a time-invariant constant,
is taken into consideration.
In the ODE \eref{ELStandard}, the dumping coefficient is set as  $\dot{\delta}(t) = D$.
$\alpha(t)$ is also set as a constant; $\dot{\alpha}(t) = 0$ and $\alpha(t) = \alpha_\text{s}$.
Base on the upper bound $\dot{\nu}(t) \leq \eexp{\alpha(t)} = \eexp{\alpha_\text{s}}$,
$\dot{\nu}(t)$ is set as its upper bound ($\dot{\nu}(t) = \eexp{\alpha_\text{s}}$).
Under this setting,
the convergence rate is $f(\bsx(t)) - f(\bsx^*) = \mathcal{O}(\eexp{-\nu(t)}) = \mathcal{O}(\exp(-\eexp{\alpha_\text{s}}t))$,
which means the linear (i.e., exponential) convergence. Therefore,
the remaining part takes the value of $\alpha_\text{s}$ as large as possible.
For a parameter $\pi(t)$, the steady state is taken into consideration and set as $\dot{\pi}(t) = 0$ and $\pi(t) = \pi_\text{s}$.
\assumptionref{para} is simplified as follows:
\begin{equation}
	\begin{cases}
		-[D - \eexp{\alpha_\text{s}}] +\eexp{\alpha_\text{s}} + \eexp{\pi_\text{s}+\alpha_\text{s}} \leq 0, \\
		[D - \eexp{\alpha_\text{s}}] -\sigma\eexp{-\alpha_\text{s}} \leq 0,                                 \\
		-[D - \eexp{\alpha_\text{s}}] - \eexp{\pi_\text{s}+\alpha_\text{s}} \leq 0.
	\end{cases}
	\elab{TICondition}
\end{equation}
The first and second items is rearranged with respect to $D$ as follows:
\begin{equation}
	2\eexp{\alpha_\text{s}} + \eexp{\pi_\text{s}+\alpha_\text{s}} \leq D \leq \eexp{\alpha_\text{s}}  + \sigma\eexp{-\alpha_\text{s}}.
\end{equation}
The third item results in $\eexp{\alpha_\text{s}} - \eexp{\pi_\text{s}+\alpha_\text{s}} \leq D$,
which is satisfied if the first inequality in the formula above holds.
The case that the two inequality above hold with equality
because the large value of left side $2\eexp{\alpha_\text{s}} + \eexp{\pi_\text{s}+\alpha_\text{s}}$ improves
the convergence rate.
Solving with respect to $\eexp{\alpha_\text{s}}$ results in $\eexp{\alpha_\text{s}} = \sqrt{\sigma/(1+\eexp{\pi_\text{s}})}$;
therefore, $\eexp{\pi_\text{s}} \to 0$ is preferred to improve the convergence rate,
and $\eexp{\alpha_\text{s}} = \sqrt{\sigma}$ is obtained.
Consequently, the coefficient is $D = 2\sqrt{\sigma}$, and the resulting convergence rate is $\mathcal{O}(\eexp{-\nu(t)}) = \mathcal{O}(\eexp{-\sqrt{\sigma}t})$.

The remaining cases of a dumping coefficient $D$, which means $D \neq 2\sqrt{\sigma}$, is taken into consideration.
For given $D$,
$\eexp{\alpha_\text{s}}$ satisfying $2\eexp{\alpha_\text{s}} \leq D \leq \eexp{\alpha_\text{s}}  + \sigma\eexp{-\alpha_\text{s}}$
can be examined by plotting the graphs of functions $y = 2x$ and $y= x + \sigma x^{-1}$.
On observing the graphs, the inequality is satisfied if $\eexp{\alpha_\text{s}} \leq \sqrt{\sigma}$.
Therefore, an arbitrary large $D$ can be chosen by taking a small $\eexp{\alpha_\text{s}}$:
\begin{equation}
	\eexp{\alpha_\text{s}} = \begin{cases}
		D/2                          & (0 < D \leq 2\sqrt{\sigma}),       \\
		[D - \sqrt{D^2 - 4\sigma}]/2 & (2\sqrt{\sigma} \leq D < +\infty).
	\end{cases}
\end{equation}
Eventually, for an arbitrary large coefficient $D$, the Lyapunov function and convergence are established
(however, $\exp(\alpha(t)) = \sqrt{\sigma}$ and $\dot{\delta}(t) = 2\sqrt{\sigma}$ is preferred from the viewpoint of the convergence rate).
The convergence rate is
\begin{equation}
	f(\bsx(t)) - f(\bsx^*) = \begin{cases}
		\mathcal{O}\left(\exp\left(-[D/2]t\right)\right)                                            & (0 < D \leq 2\sqrt{\sigma}),       \\
		\mathcal{O}\left(\exp\left(-\left[[D - \sqrt{D^2 - 4\sigma}]\middle/2\right]t\right)\right) & (2\sqrt{\sigma} \leq D < +\infty).
	\end{cases}
\end{equation}
Furthermore, the right side of the third item in Eq.~\eref{TICondition} is
\begin{equation}
	-[D-\eexp{\alpha_\text{s}}] =
	\begin{dcases}
		-D/2                          & (0 < D \leq 2\sqrt{\sigma}),       \\
		-[D + \sqrt{D^2 - 4\sigma}]/2 & (2\sqrt{\sigma} \leq D < +\infty),
	\end{dcases}
\end{equation}
and is negative. Therefore, the following results follow from \theoremref{ell2LyapunovBounds}.
\begin{corollary}
	Under \assumptionref{0}, $\bsx(t)$ of the state equation \eref{StODE} satisfies the following inequality:
	\begin{equation}
		\begin{dcases}
			\int_{t_0}^{+\infty}
			\exp\left([D/2]t\right)\frac{1}{2}\|\dot{\bsx}(t)\|_2^2\,\mathrm{d}t < + \infty & (0 < D \leq 2\sqrt{\sigma}),       \\
			\int_{t_0}^{+\infty}
			\exp\left(\big[[D + \sqrt{D^2 - 4\sigma}]/2\big]t\right)
			\frac{1}{2}\|\dot{\bsx}(t)\|_2^2\,\mathrm{d}t < + \infty                        & (2\sqrt{\sigma} \leq D < +\infty).
		\end{dcases}
	\end{equation}
\end{corollary}

\subsection{Case of Time-Varying Parameter:  Hyperbolic-Function-Based Parameter \cite{Ki23}}\subsectionlab{hyp}
This section presents the derivation of the parameters
with the hyperbolic functions \cite{Ki23}.
Parameters obtained in this section are summarized as follows:
\begin{equation}
	\begin{dcases}
		\alpha(t) =
		\begin{dcases}
			\log\left(\frac{\sqrt{\sigma}}{\tanh\left([\sqrt{\sigma}/2]t\right)}\right) & (\sigma > 0), \\
			\log\left(\frac{2}{t}\right)                                                & (\sigma = 0),
		\end{dcases}                                      \\
		\dot{\delta}(t)
		= \begin{dcases}
			  \frac{\sqrt{\sigma}\left[3+\tanh^2\left([\sqrt{\sigma}/2]t\right)\right]}{2\tanh\left([\sqrt{\sigma}/2]t\right)} & (\sigma>0),   \\
			  \frac{3}{t}                                                                                                      & (\sigma = 0),
		  \end{dcases} \\
		\eta(t) = 2\alpha(t),                                                                                                              \\
		\dot{\nu}(t) = \eexp{\alpha(t)}                                                                                                    \\
		\text{($\pi(t)$ is not used, i.e., $\pi(t)=-\infty$ or $\eexp{\pi(t)}=0$).}
	\end{dcases}
	\elab{hypParams}
\end{equation}

A unified model for both the two cases of $\sigma > 0$
and $\sigma = 0$ \cite{Ki23} is given as follows:
\begin{equation}
	\begin{dcases}
		\ddot{\bsx}(t) +
		\frac{\sqrt{\sigma}\left[3+\tanh^2\left([\sqrt{\sigma}/2]t\right)\right]}{2\tanh\left([\sqrt{\sigma}/2]t\right)}
		\dot{\bsx}(t) + \nabla f(\bsx(t)) = \bszero_{n_\bsx}            & (\sigma > 0), \\
		\ddot{\bsx}(t) +
		\frac{3}{t}\dot{\bsx}(t) + \nabla f(\bsx(t)) = \bszero_{n_\bsx} & (\sigma = 0).
	\end{dcases}
	\elab{hypODE}
\end{equation}
The state equation of $[\bsx(t),\bsz(t)]$ for $\sigma > 0$ is given by
\begin{equation}
	\begin{dcases}
		\dot{\bsx}(t) = \frac{\sqrt{\sigma}}{\tanh\left([\sqrt{\sigma}/2]t\right)}(\bsz(t) - \bsx(t)), \\
		\dot{\bsz}(t) = -\sqrt{\sigma}\tanh\left([\sqrt{\sigma}/2]t\right)(\bsz(t) - \bsx(t))
		- \frac{\tanh\left([\sqrt{\sigma}/2]t\right)}{\sqrt{\sigma}}\nabla f(\bsx(t)),
	\end{dcases}
	\elab{hypStateEq1}
\end{equation}
and that for $\sigma = 0$ is given by
\begin{equation}
	\begin{dcases}
		\dot{\bsx}(t) = \frac{2}{t}(\bsz(t) - \bsx(t)), \\
		\dot{\bsz}(t) = - \frac{t}{2}\nabla f(\bsx(t)).
	\end{dcases}
	\elab{hypStateEq2}
\end{equation}

The remaining part of this subsection illustrates the derivation of the aforementioned parameters \eref{hypParams} and the ODE \eref{hypODE}.
Based on the analysis of the time-invariant case, discussed in the previous section,
to simplify the design procedure, a term $\eexp{\pi(t)}$ is replaced with zero,
which means a term $(1/2)\|\bsx-\bsx^*\|_2^2$ is excluded in
a Lyapunov function.
Similar to the aforementioned time-invariant case, setting
$\dot{\nu}(t) = \eexp{\alpha(t)}$ is used to improve the convergence rate.
The required condition is as follows:
\begin{equation}
	\begin{cases}
		-\dot{\delta}(t) + \dot{\alpha}(t) + 2\eexp{\alpha(t)} \leq 0,                        \\
		\dot{\delta}(t) + \dot{\alpha}(t) - \eexp{\alpha(t)} -\sigma\eexp{-\alpha(t)} \leq 0, \\
		-\dot{\delta}(t) - \dot{\alpha}(t) + \eexp{\alpha(t)} \leq 0.
	\end{cases}
	\elab{hypCondition}
\end{equation}
Similar to the time-invariant case, the first and second items
is rearranged as follows:
\begin{equation}
	\dot{\alpha}(t) + 2\eexp{\alpha(t)} \leq \dot{\delta}(t) \leq -\dot{\alpha}(t) + \eexp{\alpha(t)} + \sigma\eexp{-\alpha(t)}.
\end{equation}
The third item results in
$- \dot{\alpha}(t) + \eexp{\alpha(t)} \leq \dot{\delta}(t)$,
which is satisfied if the second inequality holds with equality.
Therefore, the damping coefficient is set as $\dot{\delta}(t) = -\dot{\alpha}(t) + \eexp{\alpha(t)} + \sigma\eexp{-\alpha(t)}$.

The left side of the inequality above, which is $\dot{\alpha}(t) + 2\eexp{\alpha(t)}$,
needs to be as large as possible to improve the convergence rate  $\mathcal{O}(\eexp{-\nu(t)})$
under the setting $\dot{\nu}(t) = \eexp{\alpha(t)}$.
Therefore,
the case that the upper and lower bounds of $\dot{\delta}(t)$ are equal is taken into consideration here.
That is,
an equation $\dot{\alpha}(t) + 2\eexp{\alpha(t)} = -\dot{\alpha}(t) + \eexp{\alpha(t)} + \sigma\eexp{-\alpha(t)}$ is assumed.
The equation results in an ODE
$2\dot{\alpha}(t)\eexp{\alpha(t)} + \eexp{2\alpha(t)} - \sigma = 0$;
the hyperbolic-function-based parameter $\alpha(t)$ in Eq.~\eref{hypParams} is obtained by solving the ODE
(for detailed derivation, refer to Appendix~\appendixref{1}).
$\eexp{\alpha(t)}$ and $\dot{\alpha}(t)$ are calculated as follows:
\begin{equation}
	\begin{split}
		 & \eexp{\alpha(t)} =
		\begin{dcases}
			\frac{\sqrt{\sigma}}{\tanh\left([\sqrt{\sigma}/2]t\right)} & (\sigma > 0), \\
			\frac{2}{t}                                                & (\sigma = 0),
		\end{dcases}
	\end{split}
\end{equation}
and
\begin{equation}
	\dot{\alpha}(t) = \frac{1}{\eexp{\alpha(t)}}\dif{\eexp{\alpha(t)}}{t}
	=
	\begin{dcases}
		-\frac{\sqrt{\sigma}[1-\tanh^2\left([\sqrt{\sigma}/2]t\right)]}{2\tanh\left([\sqrt{\sigma}/2]t\right)} & (\sigma > 0), \\
		- \frac{1}{t}                                                                                          & (\sigma = 0).
	\end{dcases}
\end{equation}
As mentioned above, the coefficient of the dumping term $\dot{\delta}(t)\dot{\bsx}(t)$
is set as $\dot{\delta}(t) = -\dot{\alpha}(t) + \eexp{\alpha(t)} + \sigma\eexp{-\alpha(t)}$ and results in Eq.~\eref{hypParams}.
As pointed out in \cite{Ki23},
in the limit of $t\to+\infty$ with $\sigma > 0$,
the dumping coefficient $\dot{\delta}(t)$ converges to $2\sqrt{\sigma}$.
The substitution of the aforementioned setting
into the ODE \eref{ODEStandard} and state equation \eref{ELStandard}
results in those in Eq.~\eref{hypODE} and Eqs.~\eref{hypStateEq1} and \eref{hypStateEq2}, respectively.
The convergence rate is $f(\bsx(t)) - f(\bsx^*) = \mathcal{O}(\eexp{-\nu(t)})$.
On taking the upper bound of the condition $\dot{\nu}(t) \leq \eexp{\alpha(t)}$ (i.e., $\dot{\nu}(t) = \eexp{\alpha(t)}$),
the primitive function of $\eexp{\alpha(t)}$ is given by
\begin{equation}
	\nu(t) =
	\begin{dcases}
		2\log\left(\sinh\left([\sqrt{\sigma}/2]t\right)\right) & (\sigma > 0), \\
		2\log t                                                & (\sigma = 0).
	\end{dcases}
\end{equation}
On substituting the obtained parameters above into \theoremref{LyapunovBounds},
the convergence rate is given as follows:
\begin{equation}
	f(\bsx(t))-f(\bsx^*) \leq
	\begin{dcases}
		\frac{\sinh^2\left([\sqrt{\sigma}/2]t_0\right)}{\sinh^2\left([\sqrt{\sigma}/2]t\right)}
		\left(\frac{1}{2}\|\bsx^*-\bsz(t_0)\|_2^2+[f(\bsx(t_0))-f(\bsx^*)]\right) & (\sigma > 0), \\
		\frac{t_0^2}{t^2}
		\left(\frac{1}{2}\|\bsx^*-\bsz(t_0)\|_2^2+[f(\bsx(t_0))-f(\bsx^*)]\right) & (\sigma = 0),
	\end{dcases}
\end{equation}
where $t_0 > 0$.

In the inequalities \eref{hypCondition},
the first and second inequalities hold with equality,
and the third inequality results in $-\sigma\eexp{\alpha(t)} \leq 0$;
\theoremref{ell2LyapunovBounds} results in the following corollary.
\begin{corollary}
	Under \assumptionref{0} and $\sigma > 0$,
	$\bsx(t)$ of the state equation \eref{hypODE} satisfies the following inequality:
	\begin{equation}
		\int_{t_0}^{+\infty}
		\sqrt{\sigma}\sinh^2([\sqrt{\sigma}/2]t)\tanh([\sqrt{\sigma}/2]t)
		\frac{1}{2}\|\dot{\bsx}(t)\|_2^2\,\mathrm{d}t < +\infty.
	\end{equation}
\end{corollary}
However, in the case of $\sigma = 0$,
the third item of the inequalities \eref{hypCondition} holds with equality (i.e., $-\dot{\delta}(t) - \dot{\alpha}(t) + \eexp{\alpha(t)} = 0$);
under the equality and $\eexp{\pi(t)}=0$,
\theoremref{ell2LyapunovBounds} claims a trivial inequality $0 \leq V(t_0)$.
Handling of this issue is presented in the next subsection.
\begin{remark}
	Although the model \eref{ExistingODE} (subsection 4.1 of \cite{Ki23})
	and the model \eref{ELStandard} of this paper, derived from $\eta(t) = 2\alpha(t)$,
	have difference expressions,
	the substitution of the parameter setting of this subsection results in
	the same model \eref{hypODE}; this remark examines this behavior.
	The parameter setting is summarized as follows:
	\begin{itemize}
		\item In both \cite{Ki23} and this paper,
		      the same $\alpha(t) = \log\left(\sqrt{\sigma}/\tanh\left([\sqrt{\sigma}/2]t\right)\right)$ is used.
		\item $\beta(t)$ in \cite{Ki23}, which does not appears in this paper,
		      is set as $\beta(t) = 2\log\left(\sinh\left([\sqrt{\sigma}/2]t\right)/\sqrt{\sigma}\right)$,
		      and satisfies $\dot{\beta}(t) = \eexp{\alpha(t)}$ (the upper bound of the condition $\dot{\beta}(t) \leq \eexp{\alpha(t)}$).
		\item $\gamma(t)$, satisfying $\dot{\gamma}(t) = \eexp{\alpha(t)}$, is not explicitly shown \cite{Ki23}.
	\end{itemize}
	The setting above indicates that
	$\gamma(t)$ is equal to $\beta(t)$ except for the scaling with a constant;
	therefore, the remaining part focuses on $\beta(t)$.

	The setting of this paper $\eta(t)=2\alpha(t)$ and $h(\bsx) = (1/2)\|\bsx\|_2^2$ results in the Lagrangian
	\begin{equation}
		\mathcal{L}(\bsx,\bsv,t) = \eexp{\delta(t)}\left(\eexp{2\alpha(t)}\mathcal{D}_{(1/2)\|\cdot\|_2^2}(\bsx+\eexp{-\alpha(t)}\bsv,\bsx)-f(\bsx)\right)
		= \eexp{\delta(t)}\left(\frac{1}{2}\|\bsv\|_2^2 - f(\bsx)\right).
	\end{equation}
	For $\alpha(t) = \log\left(\sqrt{\sigma}/\tanh\left([\sqrt{\sigma}/2]t\right)\right)$,
	on using
	$\dot{\delta}(t) = \dot{\alpha}(t) + 2\eexp{\alpha(t)} = \dot{\alpha}(t) + 2\dot{\beta}(t)$
	and an integral constant $K$, the Lagrangian is expressed as
	\begin{equation}
		\begin{split}
			\mathcal{L}(\bsx,\bsv,t)
			 & = \eexp{\delta(t)}\left(\eexp{2\alpha(t)}\mathcal{D}_h(\bsx+\eexp{-\alpha(t)}\bsv,\bsx)-f(\bsx)\right)                                       \\
			 & = \eexp{K + \alpha(t)+\beta(t)}\left(\eexp{2\alpha(t)+\beta(t)}\mathcal{D}_h(\bsx+\eexp{-\alpha(t)}\bsv,\bsx)-\eexp{\beta(t)}f(\bsx)\right).
		\end{split}
		\elab{ConvertLagrangian}
	\end{equation}
	The identity of the hyperbolic functions $1/\tanh^2(\xi) = 1 + 1/\sinh^2(\xi)$, derived from $\cosh^2(\xi) - \sinh^2(\xi) = 1$,
	results in the following equation:
	\begin{equation}
		\eexp{2\alpha(t)+\beta(t)} = \frac{\sinh^2\left([\sqrt{\sigma}/2]t\right)}{\tanh^2\left([\sqrt{\sigma}/2]t\right)}
		= 1 + \frac{1}{\sinh^2\left([\sqrt{\sigma}/2]t\right)} = 1 + \sigma\eexp{\beta(t)}.
	\end{equation}
	Actually, the existing study (p.~29 of \cite{Ki23}) derives the hyperbolic-function-based parameters by
	taking into consideration the aforementioned equality $\eexp{2\alpha(t)+\beta(t)} = 1 + \sigma\eexp{\beta(t)}$ and a condition $\dot{\beta}(t) = \eexp{\alpha(t)}$.

	Consequently, the substitution of the aforementioned parameters and
	the scaling with a constant in the Lagrangian of \cite{Ki23} (Eq.~(4) of \cite{Ki23} and Eq.~\eref{ExistingLagrangian2} of this paper)
	result in the same Lagrangian \eref{ConvertLagrangian} of this paper.
\end{remark}

\subsection{Case of Time-Varying Dumping Coefficient $(C/t)\dot{\bsx}(t)$}\subsectionlab{C/t}
Without the assumption of the strong convexity ($\sigma=0$), the ODE
\begin{equation}
	\ddot{\bsx}(t) + \frac{C}{t}\dot{\bsx}(t) + \nabla f(\bsx(t)) = \bszero_{n_\bsx}
	\elab{C/tODE}
\end{equation}
is taken into consideration in this subsection.
An associated Lyapunov function has been already given in an existing study \cite{Su16} and an improved result \cite{At19},
this subsection explores parameters satisfying \assumptionref{general}
and derives a Lyapunov function of the form of \theoremref{Lyapunov2}.
Parameters obtained in this subsection is summarized as follows:
\begin{equation}
	\begin{dcases}
		\alpha(t) =	\log\left((2C)/(3t)\right), \\
		\dot{\delta}(t)
		= C/t,                                  \\
		\eta(t) = 2\alpha(t),                   \\
		\dot{\nu}(t) =
		\begin{dcases}
			(2C)/(3t) & (0 < C \leq 3),       \\
			2/t       & (3 \leq C < +\infty),
		\end{dcases}       \\
		\pi(t) =
		\begin{dcases}
			\log\left((3 -C)/(2C)\right) & (0 < C < 3),       \\
			-\infty                      & (C = 3),           \\
			\log\left((C-3)/(2C)\right)  & (3 < C < +\infty).
		\end{dcases}
	\end{dcases}
	\elab{C/tParams}
\end{equation}
The substitution of the parameters above into the state equation \eref{ELStandard},
which assumes the Bregman divergence is introduced by the $\ell_2$-norm,
results in the following
state equation of $\bsx$ and $\bsz(t) = \bsx(t) + ((3t)/(2C))\dot{\bsx}(t)$:
\begin{equation}
	\begin{dcases}
		\dot{\bsx}(t) = \frac{2C}{3t}[\bsz(t) - \bsx(t)], \\
		\dot{\bsz}(t) = -\frac{C-3}{3t}[\bsx(t) - \bsz(t)]
		- \frac{3t}{2C}\nabla f(\bsx(t)).
	\end{dcases}
	\elab{C/tStateEq}
\end{equation}

The main strategy of this subsection is that
$\dot{\nu}(t)$ needs to be as large as possible
to improve the convergence rate $\mathcal{O}(\eexp{-\nu(t)})$.
The dumping coefficient is set as $\dot{\delta}(t) = C/t$ because of the pre-assigned dumping term $(C/t)\dot{\bsx}(t)$.
To find parameters satisfying \assumptionref{general},
the form of $\alpha(t)$ is assumed to be $\eexp{\alpha(t)} = \lambda/t$ and $\dot{\alpha}(t) = -1/t$,
and that of $\nu(t)$ is set as $\eexp{\nu(t)} = t^{\lambda^\prime}$ for some constants $\lambda$ and $\lambda^\prime$.
On substituting the setting above, \assumptionref{general2} results in the following inequalities:
\begin{equation}
	\begin{cases}
		\dot{\nu}(t) = \lambda^\prime/t\leq \lambda/t,                                       \\
		(-C-1+\lambda + \lambda\eexp{\pi(t)}+\lambda^\prime)/t \leq 0,                       \\
		(C-1-\lambda)/t + (\lambda^\prime -2 +\dot{\pi}(t)t -\lambda)\eexp{\pi(t)}/t \leq 0, \\
		(-C+1+\lambda-\lambda\eexp{\pi(t)})/t \leq 0.
	\end{cases}
	\elab{ell2Condition}
\end{equation}

For $0 < C \leq 3$,
$\dot{\nu}(t) =\lambda^\prime/t \leq \eexp{\alpha(t)} = \lambda/t$ holds with equality
(i.e., $\lambda^\prime \leq \lambda$) as shown in the subsequent calculation.
The following inequalities are given:
\begin{equation}
	\begin{cases}
		(-C-1+2\lambda + \lambda\eexp{\pi(t)})/t \leq 0,          \\
		(C-1-\lambda + (-2+\dot{\pi}(t)t)\eexp{\pi(t)})/t \leq 0, \\
		(-C+1+\lambda-\lambda\eexp{\pi(t)})/t \leq 0.
	\end{cases}
\end{equation}
On taking $\lambda = 2C/3$ and $\eexp{\pi(t)} = (1 -C/3)/(2C/3)$,
the first and third inequalities hold with equality, and
the second inequality results in $(C/3 - 3/C)/t \leq 0$ and is satisfied;
therefore, the convergence rate is
$\mathcal{O}(t^{-\lambda}) = \mathcal{O}(t^{-2C/3})$.
The Lyapunov function obtained here is eventually consistent with that in \cite{At19}.
At $C=3$, $\eexp{\pi(t)}$ is replaced with zero and the same discussion follows.

The case of $3 < C < +\infty$ is not covered by the aforementioned analysis.
Therefore, a condition of the analysis above is relaxed;
the assumption that $\nu(t)$ reaches the upper bound $\dot{\nu}(t) = \eexp{\alpha(t)}$ is
replaced by an inequality $\dot{\nu}(t) \leq \eexp{\alpha(t)}$,
which results in $\lambda^\prime \leq \lambda$.
The inequalities \eref{ell2Condition} are reexamined here.

The possibility of setting $\dot{\nu}(t) = \lambda^\prime/t > 2/t$,
which means the convergence rate $f(\bsx(t)) - f(\bsx^*) = \mathcal{O}(1/t^{\lambda^\prime})\, (\lambda^\prime > 2)$,
is taken into consideration.
The rearrangement of the third item of the inequalities \eref{ell2Condition} results in
$(C-1-\lambda)/t + ([\lambda^\prime -2] +\dot{\pi}(t)t -\lambda)\eexp{\pi(t)}/t \leq 0$.
With the second item of Eq.~\eref{ell2Condition},
an inequality $\dot{\pi}(t)\eexp{\pi(t)} = \mathrm{d}\eexp{\pi(t)}/\mathrm{d}t \leq -2[\lambda^\prime -2]/t$ is obtained;
however, a solution $\eexp{\pi(t)}>0$ seems unavailable. Therefore,
$\lambda^\prime = 2$ is assumed and results in the following inequalities:
\begin{equation}
	\begin{cases}
		(-C+1+\lambda + \lambda\eexp{\pi(t)})/t \leq 0,                   \\
		(C-1-\lambda)/t  + (\dot{\pi}(t)t -\lambda)\eexp{\pi(t)}/t\leq 0, \\
		(-C+1+\lambda - \lambda\eexp{\pi(t)})/t \leq 0.
	\end{cases}
\end{equation}
On setting $\lambda = \lambda^\prime =2$ and $\eexp{\pi(t)} = (C-1 -\lambda)/\lambda = (C-3)/(2C)$,
the first and second inequalities holds with equality, and
the third inequality results in $-2(C-3)/(3t) \leq 0$; the inequalities above are fully satisfied.
Therefore, the setting $\dot{\nu}(t) = \lambda^\prime/t$ and $\lambda^\prime =2$
results in $\dot{\nu}(t) = 2/t$, $\nu(t) = 2\log t$,
and the convergence rate $\mathcal{O}(\eexp{-\nu(t)}) = \mathcal{O}(t^{-2})$.
The setting above is consistent with a Lyapunov function given in \cite{Su16}.

Consequently, the convergence rate, depending on the value of $C$, is given as follows:
\begin{equation}
	f(\bsx(t)) - f(\bsx^*) =
	\begin{dcases}
		\bigO{t^{-2C/3}} & (0 < C \leq 3),       \\
		\bigO{t^{-2}}    & (3 \leq C < +\infty).
	\end{dcases}
\end{equation}

While various inequalities similar to the form of \theoremref{LyapunovBounds} are
derived in \cite{At19}, this study illustrates the application example of \theoremref{ell2LyapunovBounds} to
$\dot{\bsx}(t)$.
\begin{corollary}
	Under \assumptionref{0}, $\bsx(t)$ of the state equation \eref{C/tODE} satisfies the following inequalities:
	\begin{equation}
		\begin{dcases}
			\int_{t_0}^{+\infty}
			t^{\lambda^\prime - 1}
			\frac{1}{2}\|\dot\bsx(t)\|_2^2\,\mathrm{d}t < +\infty                      & (0 < C \leq 3, \, \lambda^\prime < (2C)/3), \\
			\int_{t_0}^{+\infty}t\frac{1}{2}\|\dot\bsx(t)\|_2^2\,\mathrm{d}t < +\infty & (3 < C < +\infty).
		\end{dcases}
	\end{equation}
\end{corollary}
\begin{proof}
	In the case of $0 < C \leq 3$ with the parameters \eref{C/tParams},
	\theoremref{ell2LyapunovBounds} does not straightforwardly work because the forth item of Eq.~\eref{ell2Condition} holds with equality.
	Therefore, $\nu(t)$ and $\pi(t)$ are modified in this proof.
	While the aforementioned analysis set $\lambda^\prime$ as its upper bound  $\lambda^\prime = \lambda$,
	an inequality $\lambda^\prime \leq \lambda = (2C)/3$ is assumed here,
	and $\eexp{\pi(t)} = (1 -C/3)/(2C/3)$ in the analysis above is also
	relaxed as $\eexp{\pi(t)} = (1 -C/3)/(2C/3) + (\lambda - \lambda^\prime)/\lambda$,
	which takes a larger value than $\eexp{\pi(t)} = (1 -C/3)/(2C/3)$.
	The second item of Eq.~\eref{ell2Condition} holds with equality, and
	the third and forth items are satisfied by
	\begin{equation}
		\begin{dcases}
			\left(
			\frac{C}{3} - \frac{3}{C} - [\lambda-\lambda^\prime]\eexp{\pi(t)} - \frac{2[\lambda-\lambda^\prime]}{\lambda}
			\right)\frac{1}{t} < 0, \\
			-[\lambda-\lambda^\prime]/t < 0,
		\end{dcases}
	\end{equation}
	respectively. Therefore, \theoremref{ell2LyapunovBounds} is applicable and indicates that
	the first inequality of the corollary is obtained.
	The second item of the corollary follows because
	the derivation of the parameters of this subsection has already confirmed that
	the forth item of Eq.~\eref{ell2Condition} results in $-2(C-3)/(3t) \leq 0$,
	and the left side $-2(C-3)/(3t)$ is strictly negative for $0 < C < 3$;
	consequently, \theoremref{ell2LyapunovBounds} is applicable.
\end{proof}

\section{Extension to Smoothing}\sectionlab{smoothing}
This section allows a nondifferentiable objective function $f$. A slight modification to the models and the Lyapunov functions
makes the analysis of the paper applicable to
an ODE with the smooth approximation for a nondifferentiable objective function minimization.
First, the smooth approximation of a (possibly nondifferentiable) function is defined as follows (Definition 10.43 of \cite{Be17a}):
\begin{definition}\definitionlab{1}
	A convex function $h:\R^{n_\bsx}\to\R$ and
	a function $\hh{\cdot}{\mu}:\R^{n_\bsx}\times\R_{>0}\to\R$ which is convex and differentiable with respect to $\bsx$
	are considered.
	If $\hh{\cdot}{\mu}$ satisfies the conditions below for an arbitrary $\mu>0$,
	$\hh{\cdot}{\mu}$ is called a $(1/\mu)$-smooth approximation of $h$ with parameters $(\alpha,\beta)$:
	\begin{enumerate}
		\item For an arbitrary $\bsx\in\R^{n_\bsx}$,
		      an inequality $\hh{\bsx}{\mu} \leq h(\bsx) \leq \hh{\bsx}{\mu} + \beta \mu$ holds $(\beta > 0)$.
		\item $\hh{\cdot}{\mu}$ is $(\alpha/\mu)$-smooth $(\alpha > 0)$.
	\end{enumerate}
\end{definition}
A smooth approximation $\hh{\bsx}{\mu}$ of $h$ has two arguments: an original variable $\bsx$ and
a smoothing parameter $\mu$. Hereinafter,
the partial derivative with respect to these two arguments are denoted by
$\nabla_\bsx\hh{\bsx}{\mu}$ and $\nabla_\mu\hh{\bsx}{\mu}$, respectively.
The following definition, focusing on the partial derivative with respect to the smoothing parameter,
is introduced \cite{To23a}.
\begin{definition}\definitionlab{3}
	A $(1/\mu)$-smooth approximation of $h$ with parameters $(\alpha,\beta)$
	is considered. For an arbitrary $\bsx\in\R^{n_\bsx}$ and $\mu>0$, if an inequality
	\begin{equation}
		-\beta \leq \nabla_{\mu}\hh{\bsx}{\mu} \leq 0
		\elab{Definition2}
	\end{equation}
	holds, $h$ is called
	a $(1/\mu)$-Lipschitz continuous smooth approximation of $h$ with parameters $(\alpha,\beta)$.
\end{definition}

In general, $\mu(t)$ is nonincreasing and $\dot{\mu}(t) \leq 0$,
and the condition \eref{Definition2} of the Lipschitz continuous smooth approximation is exploited in the form of
\begin{equation}
	0 \leq \nabla_{\mu}\hh{\bsx}{\mu(t)}\dot{\mu}(t) \leq -\beta\dot{\mu}(t).
	\elab{DefinitionLipschitzSmooth_2}
\end{equation}

$f(\bsx)$ in the aforementioned Lagrangian \eref{Lagrangian} is replaced with the smooth approximation $\ff{\bsx}{\mu(t)}$.
The simple calculation results in the following corollary, which claims that
a simple modification of the aforementioned Lagrangian in smooth objective functions
leads to a nonsmooth extension of the Euler--Lagrange system.
\begin{corollary}\corollarylab{SmoothingEL}
	Under the same setting as \theoremref{EL},
	$f(\bsx)$ is replaced with the a $(1/\mu(t))$-smooth approximation $\ff{\bsx}{\mu(t)}$:
	\begin{equation}
		\mathcal{L}(\bsx,\bsv,t) = \eexp{\delta(t)}\left(\eexp{\eta(t)}\mathcal{D}_h(\bsx+\eexp{-\alpha(t)}\bsv,\bsx)-\ff{\bsx}{\mu(t)}\right).
		\elab{SmoothingLagrangian}
	\end{equation}
	Then, the Euler--Lagrange equation results in
	the ODE and state equation \eref{stateEq} whose gradient $\nabla f(\bsx(t))$ is replaced with $\nabla_\bsx\ff{\bsx(t)}{\mu(t)}$.
\end{corollary}

The following Lyapunov function is also given using a simple modification taking into consideration the smooth approximation.
\begin{theorem}\theoremlab{SmoothingLyapunov}
	Under Assumptions \ref{assumption:0} and \ref{assumption:general},
	the state equation \eref{stateEq} whose gradient $\nabla f(\bsx(t))$ is replaced with
	that of a $(1/\mu(t))$-Lipschitz continuous smooth approximation $\nabla_\bsx\ff{\bsx(t)}{\mu(t)}$ is
	considered. The following function $V(t)$ of $\bsx(t)$ and $\bsz(t) = \bsx(t) + \eexp{-\alpha(t)}\dot{\bsx}(t)$ in the state equation \eref{stateEq} is nonincreasing on $[t_0,T]$, which means $\dot{V}(t) \leq  0\,(t\in[t_0,T])$:
	\begin{equation}
		V(t) = \eexp{\nu(t)}\left( \eexp{\eta(t)}\mathcal{D}_h(\bsx^*,\bsz(t)) + [\widetilde{f}(\bsx(t),\mu(t)) + \beta\mu(t) - \widetilde{f}(\bsx^*,\mu(t)) ] \right).
		\elab{SmoothingLyapunov}
	\end{equation}
\end{theorem}
\begin{proof}
	In the proof of \theoremref{Lyapunov},
	the effect by replacing $f(\bsx(t)) - f(\bsx^*)$ with
	$\widetilde{f}(\bsx(t),\mu(t)) + \beta\mu(t) - \widetilde{f}(\bsx^*,\mu(t))$ is examined.
	On recalling $\ff{\bsx}{\mu}$ has two arguments,
	the time derivative of $\ff{\bsx(t)}{\mu(t)}$ is calculated as follows:
	\begin{equation}
		\begin{split}
			\dif{\ff{\bsx(t)}{\mu(t)}}{t} & = \nabla_\bsx^\top\ff{\bsx(t)}{\mu(t)}\dot{\bsx}(t) + \nabla_\mu\ff{\bsx(t)}{\mu(t)}\dot{\mu}(t)                      \\
			                              & = \eexp{\alpha(t)}\nabla_\bsx^\top \ff{\bsx(t)}{\mu(t)}[\bsz(t)-\bsx(t)] + \nabla_\mu\ff{\bsx(t)}{\mu(t)}\dot{\mu}(t) \\
			                              & \leq \eexp{\alpha(t)}\nabla_\bsx^\top \ff{\bsx(t)}{\mu(t)}[\bsz(t)-\bsx(t)] -\beta\dot{\mu}(t).
		\end{split}
		\elab{dffdt}
	\end{equation}
	In the inequality, the condition of the Lipschitz continuous smooth approximation,
	which is $\nabla_{\mu}\ff{\bsx}{\mu(t)}\dot{\mu}(t) \leq -\beta\dot{\mu}(t)$ for arbitrary $\bsx$, is used.
	Besides, the other inequality of the Lipschitz continuous smooth approximation
	$0 \leq \nabla_{\mu}\ff{\bsx}{\mu(t)}\dot{\mu}(t)$ for arbitrary $\bsx$ results in
	\begin{equation}
		- \dif{\ff{\bsx^*}{\mu(t)}}{t} = - \nabla_\mu\ff{\bsx^*}{\mu(t)}\dot{\mu}(t) \leq 0.
	\end{equation}
	To cancel out a term $-\beta\dot{\mu}(t)$ in the right side of Eq.~\eref{dffdt},
	an additional term $\eexp{\nu(t)}\beta\mu(t)$ is included in the Lyapunov function.
	By using the inequalities above, the time derivative of the Lyapunov function $\dot{V}(t)$ results in
	that of \eref{Ds2} whose right side has an additional term $\beta\dot{\nu}(t)\eexp{\nu(t)}\mu(t)$,
	which means $\dot{V}(t) \leq \beta\dot{\nu}(t)\eexp{\nu(t)}\mu(t)$.
\end{proof}
The following theorem claims that the similar formulae on the convergence rates hold
in nonsmooth convex objective functions.
\begin{theorem}\theoremlab{SmoothingLyapunovBounds}
	Under the same setting as \theoremref{SmoothingLyapunov}, the following inequalities hold on $t\in[t_0,T]$:
	\begin{equation}
		\begin{dcases}
			f(\bsx(t)) - f(\bsx^*) \leq \eexp{-\nu(t)}\left(V(t_0) + \beta\int_{t_0}^{t}\dot{\nu}(\tau)\eexp{\nu(\tau)}\mu(\tau)\,\mathrm{d}\tau\right),                       \\
			\eexp{\eta(t)}\mathcal{D}_h(\bsx^*,\bsz(t)) \leq \eexp{-\nu(t)}\left(V(t_0) + \beta\int_{t_0}^{t} \dot{\nu}(\tau)\eexp{\nu(\tau)}\mu(\tau)\,\mathrm{d}\tau\right), \\
			\int_{t_0}^{+\infty}
			\eexp{\nu(t)+\eta(t)}[\dot{\delta}(t) + \dot{\eta}(t) - \dot{\alpha}(t)-\eexp{\alpha(t)}]\mathcal{D}_h(\bsz(t),\bsx(t))\,\mathrm{d}t \leq V(t_0)
			+ \beta\int_{t_0}^{+\infty} \dot{\nu}(\tau)\eexp{\nu(\tau)}\mu(\tau)\,\mathrm{d}\tau.
		\end{dcases}
	\end{equation}
\end{theorem}
\begin{proof}
	On calculating the integrals of both the sides of $\dot{V}(t) \leq \beta\dot{\nu}(t)\eexp{\nu(t)}\mu(t)$,
	an inequality
	$V(t) - V(t_0) \leq \beta\int_{t_0}^t \dot{\nu}(\tau)\eexp{\nu(\tau)}\mu(\tau)\,\mathrm{d}\tau$ is obtained.
	An inequality
	\begin{equation*}
		\begin{split}
			V(t) & = \eexp{\nu(t)}\left( \eexp{\eta(t)}\mathcal{D}_h(\bsx^*,\bsz(t)) + [\widetilde{f}(\bsx(t),\mu(t)) + \beta\mu(t) - \widetilde{f}(\bsx^*,\mu(t)) ] \right) \\
			     & \geq  \eexp{\nu(t)}\left( \eexp{\eta(t)}\mathcal{D}_h(\bsx^*,\bsz(t)) + f(\bsx(t)) - f(\bsx^*) \right)
		\end{split}
	\end{equation*} follows from Eq.~\eref{DefinitionLipschitzSmooth_2}, obtained by \definitionref{3} of the Lipschitz continuous smooth approximation,
	and results in the first and second inequalities of the theorem.
	The formulae in the form of the integral are obtained by the similar modification of the proof of \theoremref{LyapunovBounds}.
	The inequality of $\mathcal{D}_h(\bsx^*,\bsz(t))$ is taken as an example here,
	and the following inequality follows:
	\begin{equation}
		\begin{split}
			 & -V(t_0) \leq V(t) - V(t_0)                                                                                                                     \\
			 & \leq \int_{t_0}^{t} \eexp{\nu(t)+\eta(t)}(-[\dot{\delta}(t) + \dot{\eta}(t) - \dot{\alpha}(t)-\eexp{\alpha(t)}])\mathcal{D}_h(\bsz(t),\bsx(t))
			\,\mathrm{d}t + \beta\int_{t_0}^t \dot{\nu}(\tau)\eexp{\nu(\tau)}\mu(\tau)\,\mathrm{d}\tau.
		\end{split}
	\end{equation}
	On rearranging and taking the limit $t\to+\infty$,
	the inequality is obtained (similar inequalities hold with respect to the Bregman divergence terms
	$\mathcal{D}_h(\bsx^*,\bsz(t))$ and $\mathcal{D}_h(\bsx^*,\bsx(t))$, but omitted here).
\end{proof}
The same discussion holds in the aforementioned case of the symmetric Bregman divergence.
\theoremref{SmoothingLyapunovBounds} suggests that
the effect of the smooth approximation does not appears in the convergence rate
if $\int_{t_0}^{t}\dot{\nu}(\tau)\eexp{\nu(\tau)}\mu(\tau)\,\mathrm{d}\tau$ is bounded
(e.g,
smoothing parameters $\mu(t)$ satisfying $\dot{\nu}(t)\eexp{\nu(t)}\mu(t) = \eexp{-\varepsilon t}$ and
$\dot{\nu}(t)\eexp{\nu(t)}\mu(t) = t^{-(1+\varepsilon)}$ with a positive number $\varepsilon>0$).
Particularly, for the hyperbolic-function-based parameters of \subsectionref{hyp},
equations
$\dot{\nu}(t) = \eexp{\alpha(t)} = \sqrt{\sigma}/\tanh\left([\sqrt{\sigma}/2]t\right)$ and
$\eexp{\nu(t)} = \sinh^2\left([\sqrt{\sigma}/2]t\right)$
result in
$\mu(t) = \eexp{-\varepsilon t}\tanh\left([\sqrt{\sigma}/2]t\right)/\left(\sqrt{\sigma}\sinh^2\left([\sqrt{\sigma}/2]t\right)\right)$
and
$\mu(t) = t^{-(1+\varepsilon)}\tanh\left([\sqrt{\sigma}/2]t\right)/\left(\sqrt{\sigma}\sinh^2\left([\sqrt{\sigma}/2]t\right)\right)$,
respectively, and the convergence rates are the same as those without the smoothing approximation;
however, in discretized models on the discrete-time domain,
the resulting convergence rate is generally inferior to that in a smooth objective function
because of an increasing Lipschitz-smoothness parameter and associated stepsize issues (see \cite{To23a} for further details).

\section{Concluding Remarks}
In this section, the objective function $f$ is assumed to be nondifferentiable.
This paper took into consideration
the ODEs and state equations of the continuous-time accelerated gradient methods.
A unified Lagrangian and its Lyapunov-function-based convergence analysis were presented,
and it was pointed out that
existing various models have been included in the proposed framework.
Besides, this study founds that the existing mysterious dumping coefficients are somewhat reasonable
from the viewpoint of the convergence rate.

\section{Acknowledgment}
This work was supported by JSPS KAKENHI Grant Number JP22K14279.

\appendix
\section{Derivation of Hyperbolic-Function-Based Solution of ODE}\appendixlab{1}
The derivation of $\alpha(t)$ in \subsectionref{hyp} is omitted in \cite{Ki23}.
For readers' convenience, $\alpha(t)$ is explicitly obtained based on the ODE
$2\dot{\alpha}(t)\eexp{\alpha(t)} + \eexp{2\alpha(t)} - \sigma = 0$ in this appendix.
Variable transform $\eexp{\alpha(t)} = \chi(t)$ is introduced and results in
\begin{equation}
	\dot{\chi}(t) = \frac{1}{2}(\sigma - \chi^2(t)).
	\elab{hypParaODE}
\end{equation}
An inequality $\sigma \neq \chi^2(t)$ is assumed, and the separation of variables is applied as follows:
\begin{equation}
	\frac{1}{2}\int\,\mathrm{d}t
	= \int\frac{1}{\sigma - \chi^2}\,\mathrm{d}\chi
	= \frac{1}{2\sqrt{\sigma}}\int \left(\frac{1}{\sqrt{\sigma}+\chi} + \frac{1}{\sqrt{\sigma}-\chi} \right)\,\mathrm{d}\chi.
\end{equation}
It should be noted that the case of $\chi^2(t)= \eexp{2\alpha(t)} = \sigma$ results in
the time-invariant dumping coefficient $2\sqrt{\sigma}\dot{\bsx}(t)$ of \subsectionref{TI}.
Satisfying $\chi(t) < \sqrt{\sigma}$ or $\chi(t) > \sqrt{\sigma}$
appears in the sign of the primitive function of the second term in the integrand:
\begin{enumerate}
	\item The case of $\chi(t) > \sqrt{\sigma}$:
	      $\alpha(t)$ shown in this paper is obtained.
	      The primitive function is $\log(\chi + \sqrt{\sigma})-\log(\chi -\sqrt{\sigma})$,
	      and on using a integral constant $K$,
	      an equation $\sqrt{\sigma}t + K = \log([\chi(t) + \sqrt{\sigma}]/[\chi - \sqrt{\sigma}])$ is obtained.
	      On calculating the exponential of both the sides and reducing with respect to $\chi(t)$, the solution is
	      \begin{equation}
		      \chi(t) = \sqrt{\sigma}\cdot\frac{\eexp{\sqrt{\sigma}t+K}+1}{\eexp{\sqrt{\sigma}t+K}-1}.
	      \end{equation}
	      On letting $K = 0$ and multiplying the numerator and denominator by $\eexp{-[\sqrt{\sigma}/2]t}$,
	      $\chi(t) = \eexp{\alpha(t)} = \sqrt{\sigma}/\tanh\left([\sqrt{\sigma}/2]t\right)$ is given.
	      Besides, in the limit of $K\to +\infty$, meaning
	      1. the initial value of $\chi(t)$ is set as $\sqrt{\sigma}$ and the right side of the ODE \eref{hypParaODE} becomes zero or
	      2. the integral is started from the point at infinity,
	      the limit $\chi(t) \to \sqrt{\sigma}$ holds and is consist with
	      $\eexp{\alpha(t)}=\sqrt{\sigma}$ in the time-invariant dumping coefficient $2\sqrt{\sigma}$.
	      These are equivalent to the claim of \cite{Ki23} that
	      in the limit of $t\to\infty$, the ODE \eref{hypODE}, using the hyperbolic-function-based parameters,
	      asymptotically converges to the ODE \eref{StODE}.
	\item The case of $\chi(t) < \sqrt{\sigma}$: $\chi(t) = \sqrt{\sigma}\tanh\left([\sqrt{\sigma}/2]t\right)$ is the solution;
	      however, the associated convergence rate is slower than that of aforementioned $\chi(t) > \sqrt{\sigma}$ and is not preferred.
\end{enumerate}

\end{document}